%% file: main.tex
\documentclass[leqno,12pt]{amsart} 
\usepackage{amssymb}
\usepackage{amsmath}
\usepackage{amssymb}
\usepackage{comment}
\usepackage{amsmath}
\usepackage{amsmath}
\usepackage{enumitem}
\usepackage[pdftex]{hyperref}
\usepackage{mathtools}
\usepackage{amscd}
\usepackage{wrapfig}
\usepackage{tikz}
\usepackage{ulem}
\usepackage{cleveref}

\usepackage{tikz}
\usepackage{tikz-cd}
\usetikzlibrary{positioning}
\usetikzlibrary{matrix,arrows,decorations.pathmorphing}

\usetikzlibrary{intersections}
\usetikzlibrary{calc}
\usetikzlibrary{arrows.meta}
\theoremstyle{plain} 
\newtheorem{theorem}{\indent\sc Theorem}[section]
\newtheorem{lemma}[theorem]{\indent\sc Lemma}
\newtheorem{corollary}[theorem]{\indent\sc Corollary}
\newtheorem{proposition}[theorem]{\indent\sc Proposition}
\newtheorem{claim}[theorem]{\indent\sc Claim}
\theoremstyle{definition} 
\newtheorem{definition}[theorem]{\indent\sc Definition}
\newtheorem{remark}[theorem]{\indent\sc Remark}
\newtheorem{example}[theorem]{\indent\sc Example}

\newtheorem{conjecture}[theorem]{\indent\sc Conjecture}

\newtheorem{question}[theorem]{\indent\sc Question}

\numberwithin{equation}{section}

\newcommand{\R}{\mathbb{R}}

\newcommand{\Z}{\mathbb{Z}}
\newcommand{\N}{\mathbb{N}}

\def\dim{\mathop{\mathrm{dim}}\nolimits}

\newcommand{\abs}[1]{\left\lvert#1\right\rvert}

\DeclareMathOperator{\dist}{\mathsf{dist}}

\newcommand{\calp}{\mathcal{P}}
\newcommand{\calh}{\mathcal{H}}

\newcommand{\Rp}{\R_{\geq 0}}
\newcommand{\ds}[1]{d(#1)}

\newcommand{\starNb}{\mathsf{star}}

\newcommand{\spike}{\mathsf{Spike}}

\crefname{theorem}{Theorem}{Theorems}
\crefname{lemma}{Lemma}{Lemmae}
\crefname{corollary}{Corollary}{Corollaries}
\crefname{proposition}{Proposition}{Propositions}
\crefname{definition}{Definition}{Definition}
\crefname{remark}{Remark}{Remarks}
\crefname{example}{Example}{Examples}
\crefname{section}{Section}{Sections}
\begin{document}

\title[A Combination Theorem for Geodesic Coarsely Convex Group Pairs]{A Combination Theorem for Geodesic Coarsely Convex Group Pairs} 

\author[Tomohiro Fukaya]{Tomohiro Fukaya} 

\author[Eduardo Mart\'inez-Pedroza]{Eduardo Mart\'inez-Pedroza}

\author[Takumi Matsuka]{Takumi Matsuka} 



\renewcommand{\thefootnote}{\fnsymbol{footnote}}
\footnote[0]{2020\textit{ Mathematics Subject Classification}.
Primary 20F65; Secondary 20F67, 57M07, 51F30, 58J22.}

\thanks{ 
}




\address{Tomohiro Fukaya \endgraf
Department of Mathematical Sciences \endgraf
Tokyo Metropolitan University \endgraf
Minami-osawa Hachioji Tokyo 192-0397 \endgraf
Japan
}
\email{tmhr@tmu.ac.jp}

\address{Eduardo Mart\'inez-Pedroza \endgraf
Department of Mathematics and Statistics \endgraf
Memorial University of Newfoundland \endgraf
St. John's, NL \endgraf
Canada
}
\email{emartinezped@mun.ca}

\address{Takumi Matsuka \endgraf
Department of Mathematical Sciences \endgraf
Tokyo Metropolitan University \endgraf
Minami-osawa Hachioji Tokyo 192-0397 \endgraf
Japan
}
\email{takumi.matsuka1@gmail.com}

\begin{abstract}
The first author and Oguni introduced a class of groups of non-positive
curvature, called coarsely convex group.  The recent success of the
theory of groups which are hyperbolic relative to a collection of
subgroups has motivated the study of other properties of groups from the
relative perspective.
In this article, we propose definitions for the notions of weakly
semihyperbolic, semihyperbolic, and coarsely convex group pairs
extending the corresponding notions in the non-relative case.  The main
result of this article is the following combination theorem.  Let
$\mathcal{A}_{wsh}$, $\mathcal{A}_{sh}$, and $\mathcal{A}_{gcc}$ denote
the classes of group pairs that are weakly semihyperbolic,
semihyperbolic, and geodesic coarsely convex respectively.  Let
$\mathcal A$ be one of the classes $\mathcal{A}_{wsh}$,
$\mathcal{A}_{sh}$, and $\mathcal{A}_{gcc}$.  Let $G$ be a group that
splits as a finite graph of groups such that each vertex group $G_v$ is
assigned a finite collection of subgroups $\mathcal{H}_v$, and each edge
group $G_e$ is conjugate to a subgroup of some $H\in \mathcal{H}_v$ if
$e$ is adjacent to $v$.  Then there is a non-trivial finite collection
of subgroups $\mathcal{H}$ of $G$ satisfying the following properties.
If each $(G_v, \mathcal{H}_v)$ is in $\mathcal{A} $, then
$(G,\mathcal{H})$ is in $\mathcal{A}$.  The main results of the article
are combination theorems generalizing results of Alonso and
Bridson~\cite{AB95}; and Fukaya and Matsuka~\cite{FuMa23}.
\end{abstract}

\maketitle

\tableofcontents


\input{10-Intro}
\input{15-Intro-boundary}

\section{Relative geometric action}

\input{21-relative-geometric-action.tex}


\section{Non-positively curved spaces and groups}

\input{22-NPCspaces}

\section{Criteria for geodesic coarse convexity}

\input{25-Criteria-for-gcc}

\section{Construction of equivariant trees of spaces}

\input{30-SpacesAmalgamatedProducts}

\section{Combination theorems for bicombings}

\input{40-BicombingCombination}

\section{Proofs of the main results}

\input{60-Proofs}







\section*{Acknowledgements}
The first author was supported by JSPS KAKENHI Grant number 24K06741. 
The second author thanks the other two authors for their hospitality while visiting Tokyo Metropolitan University during Spring 2023, where the ideas of this article originated. 
The second author acknowledges funding by the Natural Sciences and Engineering Research Council of Canada NSERC.
The third author was supported by JST, the establishment of university fellowships towards the creation of science technology innovations, Grant number JPMJFS2139.

\bibliographystyle{alpha}
\bibliography{ref}
\end{document}

%% file: 10-Intro.tex
\section{Introduction}
The existence of combings on Cayley graphs of finitely generated groups with particular properties has been one of the central themes in the study of finitely generated groups via geometric methods.  This theme appears with different motivations, for example, in the study of algorithmic properties of groups~\cite{ECHLPT92}, 
in generalizations of the theory of Gromov's hyperbolic groups~\cite{AB95}, or in relation with the interplay of geometry with operator K-theory~\cite{FO20}. From these references, some of the classes of finitely generated groups that appear in this context include automatic groups, semihyperbolic groups, and coarsely convex groups respectively.

The recent success of the theory of groups which are hyperbolic relative to a collection of subgroups~\cite{Gr87, Bo12, Fa99, Os06} has motivated the study of other properties of groups from the relative perspective.  The objects of study in this context are \emph{group pairs} $(G,\calh)$ consisting of a finitely generated group $G$ and a finite collection of subgroups $\calh$. In this article, we propose definitions for the notions of weakly semihyperbolic, semihyperbolic, and coarsely convex group pairs extending the corresponding notions in the non-relative case.  The main results of the article are combination theorems generalizing results of Alonso and Bridson~\cite{AB95}; and  Fukaya and Matsuka~\cite{FuMa23}. 

For a group pair, the analogous of a Cayley graph of a finitely generated group is called a Cayley-Abels graph.

\begin{definition}[Cayley-Abels graph for pairs]
A \emph{Cayley-Abels graph} of the group pair $(G,\calh)$ is a connected cocompact simplicial $G$-graph $\Gamma$ such that: 
\begin{enumerate}
    \item edge $G$-stabilizers are finite,
    
    \item vertex $G$-stabilizers are either finite or conjugates of subgroups in $\calh$,
    
    \item every $H\in \calh$ is the $G$-stabilizer of a vertex of $\Gamma$, and
    
    \item any pair of vertices of $\Gamma$ with the same $G$-stabilizer $H\in\calh$ are in the same $G$-orbit if $H$ is infinite.
\end{enumerate}
 \end{definition}

Given a group pair $(G,\calh)$, a standard example of a Cayley-Abels graph is Farb's coned-off Cayley graph $\hat\Gamma(G,\calh, S)$ where $S$ is a finite generating set of $G$, see~\cite{Fa99}. Since any two Cayley-Abels graphs of a group pair are quasi-isometric~\cite[Theorem H]{ArMP22}, their quasi-isometry properties provide invariants of group pairs. More generally, if a group pair $(G, \mathcal{H})$ acts \emph{geometrically} on a geodesic metric space $X$, then $X$ is quasi-isometric to any coned-off Cayley of the group pair; see Definition~\ref{def:rel-geometric-action} and Proposition~\ref{prop:relative-Schwartz-Milnor}.  

\begin{definition}\label{def:group_pair_is}
 \begin{enumerate}
  \item \label{def:group_pair_w-semihyp}
          A group pair $(G,\mathcal{H})$ is \emph{weakly semihyperbolic} 
          if there is a Cayley--Abels graph which admits a bounded quasi-geodesic bicombing.
  \item \label{def:group_pair_semihyp}
        A group pair $(G,\mathcal{H})$ is a \emph{semihyperbolic} if there is a Cayley--Abels graph which admits a bounded $G$-equivariant quasi-geodesic bicombing.
  \item \label{def:group_pair_cc}
        A group pair $(G,\mathcal{H})$ is a \emph{coarsely convex} 
        if there is a Cayley--Abels graph which admits a coarsely convex bicombing.
  \item \label{def:group_pair_gcc}
        A group pair $(G,\mathcal{H})$ is \emph{geodesic coarsely convex} 
        if $(G,\mathcal{H})$ acts geometrically  on a geodesic coarsely convex space $X$      
        such that
        \begin{enumerate}
        \item for every $H\in\mathcal{H}$ the fixed point set $X^H$ 
        is bounded, 
        and 
        \item for each $x\in X$ the isotropy $G_x$ is either finite 
        or conjugate to some $H\in\mathcal{H}$.
        \end{enumerate}
 \end{enumerate}
\end{definition}

Geodesic coarsely convex pairs are coarsely convex, see
Proposition~\ref{prop:relative-Schwartz-Milnor}. Note that semihyperbolic group pairs
are weakly semihyperbolic, and that coarsely convex group pairs are
weakly semihyperbolic. The fundamental question of passing from a relative property on a pair $(G,\mathcal H)$ to the corresponding absolute property on $G$ is open for our notions, for example:

\begin{question}
\label{question_CC} 
 Let $(G,\mathcal{H})$ be a (geodesic) coarsely convex group pair such that
 $\mathcal{H}$ is a finite family of infinite (geodesic) coarsely convex subgroups.  Is $G$ (geodesic) coarsely convex?
\end{question}

Let $\mathcal{A}_{wsh}$, $\mathcal{A}_{sh}$, $\mathcal{A}_{cc}$ and $\mathcal{A}_{gcc}$ denote the classes of group pairs that are weakly semihyperbolic, semihyperbolic, coarsely convex, and geodesic coarsely convex respectively. The main result of this note is the following combination theorem.

\begin{theorem}\label{thmx:general} Let $\mathcal A$ be one of the classes $\mathcal{A}_{wsh}$, $\mathcal{A}_{sh}$, and $\mathcal{A}_{gcc}$.
Let $G$ be a group that splits as a finite graph of groups such that each vertex group $G_v$ is assigned a finite collection of subgroups $\mathcal{H}_v$, and each edge group $G_e$ is conjugate to a subgroup of some $H\in \mathcal{H}_v$ if $e$ is adjacent to $v$.
 Then there is a finite collection of subgroups $\mathcal{H}$ of $G$ satisfying the following properties.
\begin{enumerate}
    \item If each $(G_v, \mathcal{H}_v)$ is in $\mathcal{A}  $, then $(G,\mathcal{H})$ is in $\mathcal{A}$.
    
    \item For any vertex $v$ and for any $g\in G_v$, the element $g$ is conjugate in $G_v$ to an element of some $Q\in\mathcal{H}_v$ if and only if $g$ is conjugate in $G$ to an element of some $H\in\mathcal{H}$.
\end{enumerate}
\end{theorem}

Observe that Theorem~\ref{thmx:general} is trivial without the second item in the conclusion; indeed, the pair $(G, \{G\})$ belongs to any of the three defined classes  since a Cayley-Abels graph for such group pair is the edgeless $G$-graph with a single vertex.

\begin{question}
 Does Theorem~\ref{thmx:general} holds for $\mathcal{A}_{cc}$?
\end{question}

The proof of the theorem follows the strategy in the work of Bigdely and
Mart\'inez-Pedroza~\cite{BiMa23} where analogous combination theorems
were proved for the class of group pairs $(G,\calp)$ that admit a well
defined relative Dehn function, and the class of group pairs $(G,\calp)$
where $\calp$ is a hyperbolically embedded collection of subgroups.

Theorem~\ref{thmx:general} is proved for  graphs of groups with a single edge, there are three cases to consider corresponding to the results below. Observe that the general case follows by induction on the number of edges of the graph.  

\begin{theorem}[Amalgamated Product]\label{thmx:Io}
Let $\mathcal A$ be one of the classes $\mathcal{A}_{wsh}$, $\mathcal{A}_{sh}$, and $\mathcal{A}_{gcc}$.  For $i\in\{
 1,2\}$, let $(G_i, \mathcal{H}_i\cup\{K_i\})$ be a pair and   $\partial_i\colon C\to K_i$  a group monomorphism. Let $G_1\ast_C G_2$  denote the amalgamated product determined by
$G_1\xleftarrow{\partial_1}C\xrightarrow{\partial_2}G_2$, and let  $\mathcal{H}=\mathcal{H}_1\cup\mathcal{H}_2$. Then:
\begin{enumerate}
     \item \label{item:Ib}  If    $( G_i,\mathcal{H}_i\cup\{K_i\} )$  is in $\mathcal A $   for each $i$, then $(G_1\ast_C G_2, \mathcal{H}  \cup\{\langle K_1,K_2 \rangle\})$  is in $\mathcal A$.  
    
\item For any $g\in G_i$, the element $g$ is conjugate in $G_i$ to an element of some $Q\in\mathcal{H}_i\cup \{K_i\}$ if and only if $g$ is conjugate in $G$ to an element of some $H\in\mathcal{H}\cup \{ \langle K_1, K_2 \rangle\}$.    
\end{enumerate}
\end{theorem}

In the following statements, for a subgroup $K$ of a group $G$ and an element $g\in G$, the conjugate subgroup $gKg^{-1}$ is denoted by $K^g$.

\begin{theorem}[HNN-extension I] \label{thmx:IIIo} 
Let $\mathcal A$ be one of the classes $\mathcal{A}_{wsh}$, $\mathcal{A}_{sh}$, and $\mathcal{A}_{gcc}$.  
Let $(G, \mathcal{H} \cup\{K,L\})$ be a pair with $K\neq L$,  $C$  a  subgroup of $K$, and   $\varphi\colon C\to L$ a group monomorphism. Let $G\ast_{\varphi}$ denote the HNN-extension
$\langle G, t\mid  t c t^{-1} =\varphi(c)~\text{for all $c\in C$} \rangle$. Then:
 \begin{enumerate}
    \item \label{item:IIIb}  If $(G, \mathcal{H}\cup\{K,L\})$ is in  $ \mathcal{A}$, then $(G\ast_{\varphi},\mathcal H \cup\{\langle K^t, L\rangle\})$ is in $\mathcal A$.
    \item For any $g\in G$, the element $g$ is conjugate in $G$ to an element of some $Q\in\mathcal{H} \cup \{K, L\}$ if and only if $g$ is conjugate in $G\ast_\varphi$ to an element of some $H\in\mathcal{H}\cup \{ \langle K^t, L \rangle\}$. 
\end{enumerate}
\end{theorem}

Observe that the second items of Theorems~\ref{thmx:Io}
and~\ref{thmx:IIIo} are direct consequences. The corollary below follows directly from Theorem~\ref{thmx:Io}, it is the same argument as in~\cite[Proof of Corollary 1.4]{BiMa23}.  

\begin{corollary}[HNN-extension II]\label{thmx:II} 
Let $\mathcal A$ be one of the classes $\mathcal{A}_{wsh}$, $\mathcal{A}_{sh}$, and $\mathcal{A}_{gcc}$.  Let $(G, \mathcal{H} \cup\{K\})$ be a pair, $C$ a subgroup of $K$, $s\in G$, and $\varphi\colon C\to K^s$ a group monomorphism. Let $G\ast_{\varphi}$ denote the HNN-extension
$\langle G, t\mid t c t^{-1} =\varphi(c)~\text{for all $c\in C$} \rangle$. Then:
\begin{enumerate}
    \item \label{item:IIbb} 
     If $(G, \mathcal{H}\cup\{K\})$ is in $\mathcal{A}$, then $(G\ast_{\varphi},\mathcal H \cup\{\langle K, s^{-1}t\rangle\})$ is in $ \mathcal A$.
     \item For any $g\in G$, the element $g$ is conjugate in $G$ to an element of some $Q\in\mathcal{H} \cup \{K\}$ if and only if $g$ is conjugate in $G\ast_\varphi$ to an element of some $H\in\mathcal{H}\cup \{ \langle K, s^{-1}t \rangle\}$.
\end{enumerate}
\end{corollary}

 The induction argument proving Theorem~\ref{thmx:general} from Theorems~\ref{thmx:Io} and~\ref{thmx:II} and Corollary~\ref{thmx:II} is the same as the argument described in~\cite[Proof of Theorem 1.1 and Remark1.5]{BiMa23}.

Quasi-isometries of group pairs have been a recent object of study~\cite{MPS22, HaHr19}. Roughly speaking they are quasi-isometries between the underlying groups that map left cosets of subgroups in the peripheral structure of the domain to left cosets of the peripheral structure of the target up to uniform finite Hausdorff distance. There is a number of properties of group pairs that are invariant under quasi-isometry including 
having a well-defined relative Dehn function~\cite{HMPS23,Os06}; having an almost malnormal peripheral structure~\cite{MPS22};  being relatively hyperbolic with respect to the peripheral structure~\cite{DS05,BDM09};    having a finite number of filtered ends (Bowditch's coends) with respect to the peripheral structure~\cite{MPS22}; and the peripheral structure having finite height or finite width~\cite{AbMa24}. 
Quasi-isometric group pairs have quasi-isometric Cayley-Abels graphs~\cite[Theorem H]{ArMP22}, and therefore  
being weakly hyperbolic and coarsely convex are quasi-isometry properties of group pairs. 


There is a notion of \emph{virtual isomorphism} for group pairs which generalizes commensurability of groups,  and virtual isomorphic group pairs are quasi-isometric, see~\cite[Def. 2.8 and Prop. 2.9]{AbMa24}. 

\begin{question}
Suppose $(G,\mathcal{H})$ and $(G',\mathcal{H'})$ are virtually isomorphic.
If $(G,\mathcal{H})$ is semihyperbolic, is
$(G',\mathcal{H'})$ semihyperbolic?  
\end{question}

This article is part of a research program aiming to generalize known
results on the coarse Baum-Connes conjecture from hyperbolic group pairs to coarsely convex group pairs. The purpose of this article is to establish a
foundation for the theory of coarsely convex group pairs and to study
their properties as a first step of the program. Below we give a brief overview of some of the short term questions that we aim to address in the future.


\subsection{Coarse Baum-Connes conjecture}
We can associate a proper metric space $Y$ with a $C^*$-algebra $C^*(Y)$,
called the Roe-algebra. Roe~\cite{MR1147350,MR1399087} 
constructed an index of a Dirac type operator on a complete
Riemannian manifold as an element of the 
$K$-group of the Roe-algebra of the manifold. In general it is very hard to
compute the $K$-groups of $C^*$-algebras. The coarse assembly map relates
the $K$-group of the Roe-algebra $C^*(Y)$ with the coarse $K$-homology 
$KX_*(Y)$ of $Y$. The coarse Baum-Connes conjecture says 
that the coarse assembly map is an isomorphism. Higson and Roe showed that
the conjecture holds for proper geodesic Gromov hyperbolic spaces.
The first author and S.Oguni showed the following two generalizations of 
Higson and Roe's results.

\begin{theorem}[{\cite{relhypgrp}}]
\label{thm:cBC-relhyp} 
 Let $(G,\mathcal{H})$ be a hyperbolic group pair where 
 $\mathcal{H}$ is a finite collection of infinite subgroups.
 If each subgroup $H\in \mathcal{H}$ satisfies
 the coarse Baum-Connes conjecture, and  admits a finite $H$-simplicial
 complex which is a universal space for proper actions, then $G$ 
 satisfies the coarse Baum-Connes conjecture.
\end{theorem}

\begin{theorem}[{\cite{FO20}}]
\label{thm:cBC-cc} 
 The coarse Baum-Connes conjecture holds for proper coarsely convex spaces.
\end{theorem}


Assuming a positive answer to Question~\ref{question_CC}, if 
$(G,\mathcal{H})$ is a (geodesic) coarsely convex group pair such that
 $\mathcal{H}$ is a finite family of infinite (geodesic) coarsely convex subgroups, then $G$ satisfies the coarse Baum-Connes conjecture. We expect the following statement to hold true.

\begin{conjecture}
\label{conj_cBC} 
 Let $(G,\mathcal{H})$ be a (geodesic) coarsely convex group pair such that
 $\mathcal{H}$ be a finite family of infinite subgroups.
 If each subgroup $H\in \mathcal{H}$ satisfies
 the coarse Baum-Connes conjecture, and  admits a finite $H$-simplicial
 complex which is a universal space for proper actions.
 Then $G$ satisfy the coarse Baum-Connes conjecture.
\end{conjecture}

In this context, the main results of this article would provide a tool to expand the class of groups that satisfy the coarse Baum-Connes conjecture.

%% file: 15-Intro-boundary.tex
\subsection{Boundaries  of group pairs}
Ideal boundaries of Gromov hyperbolic spaces play important roles 
in geometric group theory and noncommutative geometry.
In \cite{FO20}, the first author and Oguni constructed 
the ideal boundary $\partial X$
of a coarsely convex space $X$, and showed that the Euclidean 
open cone $\mathcal{O}\partial X$ is coarsely homotopy equivalent to $X$
via the ``exponential map''. 
It follows that the coarse Baum-Connes conjecture holds for coarsely convex spaces,
by applying Higson and Roe's result on the conjecture for open cones.
We remark that in the case $X$ is hyperbolic, $\partial X$ is homeomorphic 
to the Gromov boundary of $X$.

Let $G$ be a group acting geometrically on a proper coarsely convex space.
If $G$ admits a finite model for the classifying space $BG$, then by
applying Engel and Wulff's work on the combing coronas 
\cite[Corollary 7.13]{doi:10.1142/S1793525321500643}, we obtain
the following formula.
\begin{align*}
 \mathrm{cd}(G) = \dim(\partial G)+1.
\end{align*}
For details, see \cite[Corollary 8.10.]{FO20}.
If $G$ is hyperbolic, this formula is due to Bestvina and Mess~\cite{MR1096169}.

Let $(G,\calh)$ be a geodesic coarsely convex group pair and let 
$X$ be a geodesic coarsely convex space such that $G$ acts on $X$ 
as in item (\ref{def:group_pair_gcc}) of \cref{def:group_pair_is}.

In the case $(G,\calh)$ is hyperbolic, the Gromov boundaries of $X$ and the Cayley-Abels graph of the pair $(G,\calh)$ are homeomorphic~\cite{Gr87}, and this space is called 
the \emph{Bowditch boundary} and we can denoted by $\partial (G,\mathcal H)$.
We remark that for geodesic coarsely convex group pairs this is not the case, Croke and Kleiner found a group acting geometrically on two CAT(0) spaces with non-homeomorphic ideal boundaries~\cite{CK00}.


Manning and Wang \cite[Theorem 5.1.]{Cohom-Bowd-bdry} proved the following.
Let $(G,\calh)$ be a hyperbolic group pair of type $F$. Suppose further that
$\mathrm{cd}(G)< \mathrm{cd}(G,\calh)$. Then,
\begin{align}
    \label{formula_cohom-dim}
 \mathrm{cd}(G,\calh) = \dim(\partial(G,\calh)) +1.
\end{align}

\begin{question}
  Let $(G,\calh)$ be a (geodesic) coarsely convex group pair of type $F$ with a geometric action on a  (geodesic) coarsely convex space $X$. 
  Suppose that
  $\mathrm{cd}(G)< \mathrm{cd}(G,\calh)$. 
    Does $\mathrm{cd}(G,\calh) = \dim(\partial X) +1$ holds? 
\end{question}


%% file: 21-relative-geometric-action.tex

In this section, we introduce the notion of relative properly discontinuous action  on a geodesic metric space, the notion of a relative geometric action, and we prove a relative version of the Schwarz-Milnor lemma.  For the rest of the section, let $G$ be a group and $\mathcal{H}$ be a finite collection of subgroups, and $X$ a geodesic metric space with a $G$-action by isometries. 

\subsection{Relative properly discontinuous action}

\begin{definition}
\label{def:rel-p.d.actionII}
 We say that a subset $V\subset G$ 
 has a \emph{finite diameter relative to $\mathcal{H}$}, if 
 there exist a finite subset $S\subset G$ such that the diameter of $V$
 measured by the word metric with respect to the set 
 $S\cup \bigcup\mathcal{H}$ is finite. 

 The $G$-action on $X$ is 
 \emph{properly discontinuous relative to $\mathcal{H}$} 
 if for every bounded open subset $U\subset X$,  
 the set $\{g\in G\colon U\cap g.U\neq \emptyset\}$
 has a finite diameter relative to $\mathcal{H}$.
\end{definition}

\begin{remark}
 In the above definition, we do not assume that the set 
 $S\cup \bigcup \mathcal{H}$ generate $G$. So if $g\in G$ is not included in
 the subgroup generated by $S\cup \bigcup \mathcal{H}$, we consider the word length
 of $g$ is infinite. 
 
 We also consider the coned-off Cayley graph $\hat\Gamma(G,\mathcal{H},S)$,
 which is not necessarily connected.
 The word metric with respect to
 the set $S\cup \bigcup\mathcal{H}$ is bi-Lipschitz equivalent to 
 the graph metric $\dist_{\hat\Gamma}$ on $\hat\Gamma(G,\mathcal{H},S)$.
\end{remark}

\begin{lemma}
\label{lem:criteria-p.d.action}
 The action of $G$ on $(X,d_X)$
 is properly discontinuous relative to 
 $\mathcal{H}$ if and only if there exists $x_*\in X$ such that 
 for all $r>0$, the set $\{g\in G\mid d_X(x_*, g.x_*)\leq r\}$ 
 has a finite diameter relative to $\mathcal{H}$.
\end{lemma}

\begin{proof}
 Let $U\subset X$ be a bounded open subset. There exists $r>0$ such 
 that $U\subset B(x_*,r)$, where $B(x_*,r)$ denotes the open ball of 
 radius $r$ centered at $x_*$. Let $g\in G$. If 
 $U\cap g.U \neq \emptyset$, then we have $d_X(x_*, g.x_*)\leq 2r$.
 Therefore, by the assumption, the set 
 $\{g\in G\colon U\cap g.U\neq \emptyset\}$
 has a finite diameter relative to $\mathcal{H}$.
\end{proof}

\subsection{Relative geometric action}

\begin{definition}
\label{def:rel-geometric-action}
The $G$-action on $X$ is \emph{geometric relative to
$\mathcal{H}$} if it is properly discontinuous 
relative to
$\mathcal{H}$, is cocompact, and for every $H\in \mathcal{H}$ the fixed
point set $X^H$ is non-empty. In this case, we say that the group pair $(G,\mathcal{H})$ \emph{acts geometrically} on $X$.
\end{definition}

The following example illustrates that for a relative geometric action on a geodesic space, the fixed point sets of the peripheral subgroups are not necessarily bounded. This condition is required in the definition of geodesic coarsely convex group pair.

\begin{example}
 Consider the group  $G=\Z\times\Z\times \Z$ and the subgroups $H=\Z\times\Z\times \{0\}$ and $K=\Z\times\{0\}\times \{0\}$. The action of $G$ on the coned-off Cayley graph $X=\hat\Gamma(G,\{H,K\})$ is geometric relative to the subgroup $H$. This example illustrates that
 \begin{enumerate}[label=$($\alph*$)$]
 \item the fixed point set $X^H$ is not bounded; and 
\item the $G$-stabilizer of any vertex of $X$ of the form $gK$ is neither finite nor a conjugate of $H$.  
\end{enumerate}
\end{example}



\begin{corollary}
Let $G$ be a group, $\mathcal H$ be a finite collection of subgroups,
and $S$ a finite (relative) generating set of $G$ (with respect to
$\mathcal H$). The $G$-action on the coned-off Cayley graph
$\hat\Gamma(G,\mathcal{H},S)$ is geometric relative to $\mathcal H$
\end{corollary}



\subsection{A Relative Schwarz-Milnor lemma}

\begin{lemma}\label{lem:RelativeProperAction}
 If $G$ acts geometrically on $X$ relative to $\mathcal{H}$, 
 then there is a bounded open subset $U$ such that $G.U=X$, and in particular  
 $\{g\in G\colon U\cap g.U\neq \emptyset\}$ 
 has a finite diameter relative to $\mathcal{H}$.
\end{lemma}
\begin{proof}
 Since the $G$-action on $X$ is cocompact, 
 there exists a compact subset  $K$ of $X$ such that $G.K=X$.
 Since the action is properly discontinuous relative 
 to $\mathcal{H}$, for each $x\in K$ there is a bounded 
 open subset $U_x$ such that 
 the set $\{g\in G\colon U_x\cap g.U_x\neq\emptyset\}$ 
 has a finite diameter relative to $\mathcal{H}$.
 By compactness, $K$ is contained in a  
 finite union $U=U_{x_1}\cup \cdots \cup U_{x_\ell}$. 
 Since each $U_{x_i}$ is bounded, $U$ is bounded.  
\end{proof}

\begin{proposition}[Relative Schwarz-Milnor lemma]\label{prop:relative-Schwartz-Milnor}
Let $G$ be a group and $\mathcal{H}$ a finite collection of subgroups. If $G$ acts geometrically relative to $\mathcal{H}$ on a geodesic metric space $X$ then:
\begin{enumerate}
    \item $G$ is finitely generated relative to $\mathcal{H}$; and
    \item if $S$ is a finite relative generating set of $G$ with respect to $\mathcal{H}$, then there is a quasi-isometry between the coned-off Cayley graph 
          $\hat\Gamma(G,\mathcal{H},S)$ and $X$. 
\end{enumerate}
\end{proposition}
\begin{proof}
Let $x\in X$ and let $U_r$ be an open ball of radius $r$ and center $x$ in $X$ satisfying the conclusion of Lemma~\ref{lem:RelativeProperAction}. Hence, any ball $U_R$ of radius with $R\geq r$ and center $x$ satisfies $G. U_R =X$. Without loss of generality, assume that $r>1$.
Since the $G$-action is properly discontinuous relative to $\mathcal{H}$, there exist a finite subset $S\subset G$ such that the diameter of 
the set
\begin{align*}
 V=\{g\in G\colon U_{2r}\cap g.U_{2r} \neq \emptyset\}
\end{align*}
measured by the word length with respect to $S\cup \bigcup \mathcal{H}$
is finite. We claim that $S$ is a finite relative generating set of 
$G$ with respect to $\mathcal{H}$.

Consider the coned-off Cayley graph $\hat\Gamma(G,\mathcal{H},S)$ with  the graph metric $\dist_{\hat\Gamma}$.
The diameter of $V$ measured by $\dist_{\hat\Gamma}$ is also finite,
so we let $N$ be this value.

Now we mimic the proof of the Schwarz-Milnor's lemma. 
Let $g\in G$. If $g.x=x$ then $g\in V$ by definition. In particular, $g$ is a
product of elements of $S$ and $\bigcup \mathcal H$; moreover, note that
$\dist_{\hat\Gamma}(1,g) \leq N$ if $g\in V$.  
Otherwise, suppose $x\neq g.x$ 
and consider a geodesic $\gamma\colon [0,1]\to X$ between $x$ and
$g.x$. Then there is a finite sequence of real numbers
$t_0=0<t_1<\cdots<t_\ell=1$ and a sequence $1=f_0,f_1,\ldots , f_\ell=g$
of elements of $G$ such that $\dist_X(\gamma(t_i),\gamma(t_{i+1}))<r$
and $\gamma(t_i) \in f_i.U_r$. 
Note that we can assume that $\ell\leq \frac{2}{r}\dist_X(x,g.x) + 1$. 
In particular, $f_i.U_{2r}\cap f_{i+1}.U_{2r} \neq \emptyset$ and hence 
$f_i^{-1}f_{i+1}$ belongs to $V$ and 
$\dist_{\hat\Gamma}(1,f_i^{-1}f_{i+1}) 
= \dist_{\hat\Gamma}(f_i,f_{i+1})\leq N$.


Since $f_0=1$  and $f_\ell=g$, 
it follows that $g$ is a product of elements in $S$ and $\bigcup \mathcal{H}$. %
Specifically, we have  
\begin{align*}
  g= (f_1^{-1}f_{2}) \cdots (f_{(l-1)}^{-1}f_{l})
\end{align*}
which in $\hat\Gamma(G,\mathcal{H},S)$ describes a path of length 
at most $N(\ell-1)$ from $1$ to $g$. In other word, we have
\begin{align*}
 \dist_{\hat\Gamma}(1, g) \leq \sum_{i=1}^{l-1}\dist_{\hat\Gamma}(f_i,f_{i+1})
 \leq N(l-1)
\end{align*}

Therefore,
\[ \dist_{\hat\Gamma}(1, g) \leq \frac{2N}{r} \dist_X(x, g.x),\]
if $x\neq g.x$. In general, 
\[ \dist_{\hat\Gamma}(1, g) \leq \frac{2N}{r} \dist_X(x, g.x)+N,\]
for any $g\in G$.

For each $H\in \mathcal{H}$, let $x_H\in X$ be a point with $G$-stabilizer equal to $H$. Let $\Phi\colon \hat\Gamma(G,\mathcal{H},S) \to X$ be given by $g\mapsto g.x$ and $gH\mapsto g.x_H$. Let $\lambda=\max\{\dist_X(x,s.x) \colon s\in S\}\cup \{\dist_X(x,x_H)\colon H\in \mathcal{H}\}.$ Then it is an observation that 
\[ \dist_X (x,g.x) \leq \lambda \dist_{\hat\Gamma}(1,g). \]

Since $U_r$ is a ball of radius $r$, $\bigcup_{g\in G} g.U_R=X$ and $\Phi(1)=x\in U_r$, it follows that $\Phi$ is a quasi-isometry.   
\end{proof}

The proof of the above proposition has the following consequence.  
\begin{corollary}
If $G$ acts geometrically relative to a finite collection of subgroups
$\mathcal{H}$ on a geodesic metric space $X$, then $G$-orbits in $X$ are
discrete.
\end{corollary}


The following lemma is used in the proof of 
\cref{thm:CombinationFine}~(\ref{item:Amalgamation-relative-p.d}).


\begin{lemma}
\label{lem:rel-p.d.-coarse-embedding}
 Suppose the action of $G$ on a metric space $X$ is 
 proper discountiuous relative to $\mathcal{H}$.
 Let $S\subset G$ be a finite
 relative generating set of $G$ and let $\hat\Gamma(G,\mathcal{H},S)$ be 
 the coned-off Cayley graph. We fix a base point $x_*\in X$.

 Then there exists a non-decreasing 
 function $\Lambda=\Lambda[x_*,S]\colon \R_{\geq 0}\to \N$
 such that for $x\in X$ and $g\in G$, 
 we have
 \begin{align*}
  \dist_{\hat \Gamma}(1,g)\leq \Lambda[x_*,S](2d_X(x_*,x) + d_{X}(x,g.x))
 \end{align*}
\end{lemma}

\begin{proof}
 For $r\geq 0$, the set 
 $\{g\in G \mid \bar{B}(x_*,r)\cap g.\bar{B}(x_*,r)\neq \emptyset\}$
 has a finite diameter relative to $\mathcal{H}$. Here $\bar{B}(x_*,r)$
 denote the closed ball of radius $r$ centerd at $x_*$.
 Thus, we define
 \begin{align*}
  \Lambda(r)=\Lambda[x_*,S](r) = \max\{\dist_{\hat \Gamma}(1,g)\mid 
  g\in G,\: \bar{B}(x,r)\cap g.\bar{B}(x,r)\neq \emptyset
  \}
 \end{align*}
 For $x\in X$ and $g\in G$, we have
 \begin{align*}
  d_X(x_*,g.x_*)\leq d_{X}(x_*,x) + d_{X}(x,g.x) + d_{X}(g.x,g.x_*)
  = 2d_X(x_*,x) + d_{X}(x,g.x)
 \end{align*}
 Set $r_{x,g}= 2d_X(x_*,x) + d_{X}(x,g.x)$. 
 Then $\bar{B}(x_*,r_{x,g})\cap g.\bar{B}(x_*,r_{x,g}) \neq \emptyset$. 
 It follows that
\[    \dist_{\hat \Gamma}(1,g) \leq \Lambda[x_*,S](r_{x,g})   \qedhere
\]
\end{proof}

%% file: 22-NPCspaces.tex
\subsection{(weakly) semihyperbolic groups}
Let $(X,d_X)$ be a metric space.
For $\lambda \geq1$ and $k \geq 0$,   a $(\lambda,k)$-\textit{quasi-geodesic bicombing} on $X$ is a map $\gamma : X \times X \times [0,1] \to X$ 
such that for $x,y \in X$, the map $\gamma(x,y)(\cdot):[0,1] \to X$ satisfies $\gamma(x,y)(0)=x$, $\gamma(x,y)(1)=y$, and 
\begin{align*} 
\lambda^{-1}|t-s|d_X(x,y) - k \leq d_X(\gamma(x,y)(t), \gamma(x,y)(s)) \leq \lambda |t-s|d_X(x,y) +k \ \ \ \ (t,s \in [0,1]). 
\end{align*}
In the case that for all $x,y\in X$, the path $\gamma(x,y)(\cdot):[0,1] \to X$ satisfies 
\begin{align*}
d_X(\gamma(x,y)(t), \gamma(x,y)(s)) = |t-s| d_X(x,y) \ \ \ \ (t,s \in [0,1]), 
\end{align*}
we say that   $\gamma : X \times X \times [0,1] \to X$ is a \textit{geodesic bicombing}. 
If the space  $X$ admits a $(\lambda,k)$-quasi-geodesic bicombing for some $\lambda \geq 1$ and $k>0$,
we say that $X$ admits a quasi-geodesic bicombing. 

\begin{definition}\label{def:EquivBicombing}
Let $X$ be a geodesic metric space with a $G$-action by isometries, and a geodesic bicombing $\gamma \colon X \times X \times [0,1] \to X$. We shall say that $\gamma$ is $G$-equivariant if     
\begin{align*}
g \cdot \gamma(x,y)(t)=\gamma(gx,gy)(t)
\end{align*}
for any $g \in G$,   $x,y \in X$, and  $t \in [0,1]$.
\end{definition}

Gromov introduced the notion of non-positively curved groups, called semihyperbolic groups.
Alonso and Bridson formulated an appropriate definition of semihyperbolicity. 
We review semihyperbolicity defined by Alonso and Bridson.
\begin{definition}
A $(\lambda,k)$-quasi-geodesic bicombing $\gamma : X \times X \times [0,1] \to X$ is a $(\lambda,k,c_1,c_2)$-\textit{bounded bicombing} 
if there exist $c_1 \geq 1$ and $c_2 \geq 0$ such that for any $x_1, x_2, y_1, y_2\in X$ and for any $t \in [0,1]$,
\begin{align*} 
d_X(\gamma(x_1,y_1)(t), \gamma(x_2,y_2)(t)) \leq c_1 \max \{d_X(x_1,x_2), d_X(y_1,y_2)\} + c_2
\end{align*}
holds.  
%
A metric space is a \textit{semihyperbolic space} if it admits a $(\lambda,k,c_1,c_2)$-bounded bicombing for some constants $\lambda,k,c_1,c_2$.
\end{definition}

Semihyperbolic groups have several algebraic properties. 
See \cite{AB95} for details. 

\subsection{Coarsely convex spaces and groups}
Fukaya and Oguni \cite{FO20} introduced the class of coarsely convex spaces as a coarse geometric analogue of the class of nonpositively curved Riemannian manifolds.
Being coarsely convex is invariant under coarse equivalence and closed under direct products. 
In general, this important property does not hold in CAT(0) spaces. 
We reformulate coarsely convexity by using the notion of bicombings. 
\begin{definition}
\label{def:ccbicombing}
Let $\lambda\geq 1$, $k \geq 0$, $E\geq 1$,
and $C\geq 0$ be constants. 
Let $\theta\colon \Rp\to\Rp$ be a non-decreasing function.
A $(\lambda,k,E,C,\theta)$
-\textit{coarsely convex bicombing} 
on a metric space $(X,d_X)$ is a $(\lambda,k)$-quasi-geodesic bicombing 
$\gamma \colon X \times X \times [0,1] \to X$ with the following:
\begin{enumerate}
 \item \label{qconvex}
       Let $x_1,x_2,y_1,y_2 \in X$ and let $a,b \in [0,1]$.
       Set $y_1':=\gamma(x_1,y_1)(a)$ and $y_2':=\gamma(x_2,y_2)(b)$.
       Then, for $c \in [0,1]$, we have
 \begin{align*} 
  d_X(\gamma(x_1,y_1)(ca), \gamma(x_2,y_2)(cb)) 
  \leq (1-c)Ed_X(x_1,x_2) + cEd_X(y_1',y_2') + C.
 \end{align*}
 \item \label{qparam-reg}       
       Let $x_1,x_2,y_1,y_2 \in X$.
       Then for $t,s\in [0,1]$ 
       we have
       \begin{align*}
	\abs{t\ds{x_1,y_1} - s\ds{x_2,y_2}} \leq 
        \theta(\ds{x_1,x_2}+\ds{\gamma(x_1,y_1,t),\gamma(x_2,y_2,s)}).
       \end{align*}
\end{enumerate} 
We say that $X$ is a coarsely convex space if
it admits a $(\lambda,k,E,C,\theta)$
-coarsely convex bicombing
for some $\lambda \geq 1$, $k \geq 0$, $E \geq 1$, $C>0$, and $\theta$.
We call a group which act on a coarsely convex space geometrically a \emph{coarsely convex group}.
\end{definition}

In this paper, we mainly consider geodesic bicombings. In this case,
we can omit (\ref{qparam-reg}) of \cref{def:ccbicombing} by
the triangle inequality.

\begin{lemma}
\label{lem:cc3-tri} 
 Let $(X,d_X)$ be a metric space.  Let
 $\gamma \colon X \times X \times [0,1] \to X$ 
 be a geodesic bicombing on $X$. Then for 
 $x_1,x_2,y_1,y_2 \in X$ and $t,s \in [0,1]$, we have
 \begin{align*}
  \abs{t d_X(x_1,y_1) -s d_X(x_2,y_2)}
  \leq d_X(x_1,x_2) + d_X(\gamma(x_1,y_1)(t),\gamma(x_2,y_2)(s)).
 \end{align*}
\end{lemma}

\begin{proof}
 We suppose that $t>s$.
 Since $\gamma(x_1,y_1,-)$ and $\gamma(x_2,y_2,-)$ are geodesics, we have
 \begin{align*}
  |t d_X(x_1,y_1) -s d_X(x_2,y_2)|
  =d_X(\gamma(x_1,y_1)(0),\gamma(x_1,y_1)(t))
  -d_X(\gamma(x_2,y_2)(0),\gamma(x_2,y_2)(s)).
 \end{align*}
 By the triangle inequality, we have
 \begin{align*}
  &d_X(\gamma(x_1,y_1)(0),\gamma(x_1,y_1)(t))
  -d_X(\gamma(x_2,y_2)(0),\gamma(x_2,y_2)(s)) \\
  &\leq d_X(\gamma(x_1,y_1)(0),\gamma(x_2,y_2)(0)) + d_X(\gamma(x_2,y_2)(0),\gamma(x_2,y_2)(s)) \\
  &\ +d_X(\gamma(x_2,y_2)(s),\gamma(x_1,y_1)(t)) -  d_X(\gamma(x_2,y_2)(0),\gamma(x_2,y_2)(s)) \\
  &=d_X(x_1,x_2) + d_X(\gamma(x_1,y_1)(t),\gamma(x_2,y_2)(s)).
 \end{align*}
 This completes the proof.
\end{proof}

\begin{definition}
\label{gccbicombing}
An $(E,C)$-\textit{geodesic coarsely convex bicombing} on a metric space $(X,d_X)$ 
is a geodesic bicombing $\gamma : X \times X \times [0,1] \to X$ with the following properties:
Let $x_1,x_2,y_1,y_2 \in X$ and let $a,b \in [0,1]$.
We set $\gamma(x_1,y_1)(a)=y_1'$ and $\gamma(x_2,y_2)(b)=y_2'$.
Then, for $c \in [0,1]$, 
\begin{align*} 
d_X(\gamma(x_1,y_1)(ca), \gamma(x_2,y_2)(cb)) \leq (1-c)Ed_X(x_1,x_2) + cEd_X(y_1',y_2') + C
\end{align*}
holds.

We say that $X$ admits a geodesic coarsely convex bicombing if
it admits a $(E,C)$-geodesic coarsely convex bicombing for some $E \geq 1$ and $C>0$.
We also say that $X$ is geodesic coarsely convex if $X$ admits a geodesic coarsely convex bicombing. 
\end{definition}

There are a number of examples of geodesic coarsely convex spaces including geodesic Gromov hyperbolic spaces, CAT(0) spaces, Busemann spaces, systolic complexes~\cite{bdrySytolic},  proper injective metric spaces~\cite{convex-bicomb}, and Helly graphs~\cite{chalopin2020helly-arXiv}. Examples of geodesically coarsely convex groups arise as groups acting geometrically on these spaces, for example, weak Garside groups of finite type and Artin groups of type FC which are Helly~\cite{MR4285138},  
Artin groups of large type~\cite{MR3999678}, and graphical $C(6)$ small cancellation groups~\cite{MR3845715} which  are systolic.

\begin{lemma}
\label{lem:gccb-is-qgccb}
 An $(E,C)$-\textit{geodesic coarsely convex bicombing} $\gamma$
 on a metric space $(X,d_X)$ is a 
 $(1,0,E,C,\textrm{id}_{\Rp})$-\textit{coarsely convex bicombing} 
 on $(X,d_X)$.
\end{lemma}

\begin{proof}
 By \cref{lem:cc3-tri}, $\gamma$ satisfies (\ref{qparam-reg})
 of \cref{def:ccbicombing} for $\theta$ being the identity.
\end{proof}

\begin{lemma}  \label{gccc}
Let $E \geq 1$ and $C>0$. 
We suppose that a geodesic bicombing $\gamma : X \times X \times [0,1] \to X$ satisfies the following properties:
\begin{enumerate}

\item \label{gccc1}
Let $x,y_1,y_2 \in X$ and let $a,b \in [0,1]$.
Set $y_1'=\gamma(x,y_1)(a)$ and $y_2'=\gamma(x,y_2)(b)$.
For any $c \in [0,1]$, we have
\begin{align*} 
d_X(\gamma(x,y_1)(ca), \gamma(x,y_2)(cb)) \leq cEd_X(y_1',y_2') + C.
\end{align*}

\item \label{gccc2}
Let $x_1, x_2, y \in X$ and let $a \in [0,1]$. 
Set $y'=\gamma(x_1,y)(a)$.
For any $c \in [0,1]$, we have 
\begin{align*} d_X(\gamma(x_1,y)(ca), \gamma(x_2,y')(c)) \leq (1-c)Ed_X(x_1,x_2)+C.
\end{align*}
   
\end{enumerate}
Then $\gamma$ is a $(E,2C)$-geodesic coarsely convex bicombing. 
\end{lemma}

\begin{proof}
Let $x_1,x_2,y_1,y_2 \in X$ and let $a,b \in [0,1]$.
We set $\gamma(x_1,y_1)(a)=y_1'$ and $\gamma(x_1,y_2)(b)=y_2'$.
By assumption (\ref{gccc1}), we have that
\begin{align*} 
d_X(\gamma(x_1,y_1)(ca), \gamma(x_1,y_2')(c)) \leq cEd_X(y_1',y_2') + C.
\end{align*}
By assumption (\ref{gccc2}), we have that
\begin{align*} 
d_X(\gamma(x_2,y_2)(cb), \gamma(x_1,y_2')(c)) \leq (1-c)Ed_X(x_1,x_2)+C.
\end{align*}
By combining these equations, 
\[  
d_X(\gamma(x_1,y_1)(ca), \gamma(x_2,y_2)(cb)) \leq (1-c)Ed_X(x_1,x_2) + cEd_X(y_1',y_2') + 2C. \qedhere
\]
\end{proof}

\begin{definition}
    Let $K\geq 0$.
    A geodesic bicombing $\gamma$ is $K$-\textit{coarsely consistent} if 
    there exists $K \leq 0$ such that
    for all $x,y \in X$ and $a\in [0,1]$, if $z=\gamma(x,y)(a)$ then  
    \begin{align*}
        d_X(\gamma(x,z)(c),\gamma(x,y)(ca)) \leq  K 
    \end{align*}
    for all $c\in [0,1]$.         
    In the case that we can take $K$ as 0, we call
    $\gamma$ \textit{consistent}.
\end{definition}

\begin{definition}
    Let $X$ be a metric space with a geodesic bicombing $\gamma$.
    Let $E\geq 1$ and $C\geq 0$ be constants.
 We say that $\gamma$ is $(E,C)$-\textit{coarsely forward convex} if 
 for any $v,w_1,w_2 \in X$ and $c\in [0,1]$,  
  \begin{align*}
    d_X(\gamma(v,w_1)(c), \gamma(v,w_2)(c)) \leq cEd_X(w_1, w_2)+C.
\end{align*}

Analogously,  $\gamma$ is $(E,C)$-\textit{coarsely backward convex} if 
 for any $v_1,v_2,w \in X$ and $c\in [0,1]$,  
 \begin{align*}
    d_X(\gamma(v_1,w)(c), \gamma(v_2,w)(c)) \leq (1-c)Ed_X(v_1, v_2)+C.
\end{align*}
\end{definition}

\begin{proposition}  \label{ccgccc}
Let $E \geq 1$ $C\geq 0$ and $K\geq 0$ and let  $\gamma : X \times X \times [0,1] \to X$ is a geodesic bicombing on a metric space $X$.   If $\gamma$ is $K$-coarsely consistent, $(E,C)$-coarsely forward convex, and $(E,C)$-coarsely backward convex, then $\gamma$ is a geodesic $(E,2C+4K)$-coarsely convex bicombing. 
\end{proposition}


\begin{proof}
    It is enough to show that $\gamma$ satisfies item (\ref{gccc1}) and
    (\ref{gccc2}) in Lemma \ref{gccc}. 
    We give a proof of (\ref{gccc1}). 
 Let $x,y_1,y_2 \in X$ and let $a,b \in [0,1]$.
 Set $y_1'=\gamma(x,y_1)(a)$ and $y_2'=\gamma(x,y_2)(b)$.
 Let $c\in [0,1]$.
 Since $\gamma$ is $K$-coarsely consistent, we have
\[ 
     d_X(\gamma(x,y_1')(c),\gamma(x,y_1)(ca)) \leq K \quad \text{and} \quad
     d_X(\gamma(x,y_2')(c),\gamma(x,y_2)(cb)) \leq K
\]
It follows that
 \begin{align*}
     d_X(\gamma(x,y_1)(ca), \gamma(x,y_2)(cb))
     \leq & \  d_X(\gamma(x,y_1)(ca),\gamma(x,y_1')(c))\\
       & + d_X(\gamma(x,y_1')(c), \gamma(x,y_2')(c))\\
       & + d_X(\gamma(x,y_2')(c),\gamma(x,y_2)(cb))\\
     \leq & \ cEd_X(y_1',y_2') + C +2K.
 \end{align*}

Now we prove  \eqref{gccc2}. Let $x_1, x_2, y \in X$ and let $a \in [0,1]$. 
Set $y'=\gamma(x_1,y)(a)$.
Let $c \in [0,1]$. By $K$-coarsely consistency, 
\begin{align*}
 d_X(\gamma(x_1,y)(ca),\gamma(x_1,y')(c)) \leq K.
\end{align*}
Then we have
 \begin{align*} 
  d_X(\gamma(x_1,y)(ca), \gamma(x_2,y')(c)) \leq & \
  d_X(\gamma(x_1,y)(ca),\gamma(x_1,y')(c)) \\
  & + d_X(\gamma(x_1,y')(c),\gamma(x_2,y')(c)) \\
  \leq & \ (1-c)Ed_X(x_1,x_2) + C + K.\qedhere
 \end{align*}
\end{proof}

\subsection{Geodesic coarsely convex group pairs}
In this subsection, we give an example of geodesic coarsely convex group pairs. 
Let $F_2 = \langle a, b \rangle$ be the free group of rank $2$
and let $S=\{a,b\}$ be the standard generating set. 
For a subgroup $A= \langle a \rangle$, 
we consider the coned-off Cayley graph $\hat{\Gamma}(F_2\times \mathbb{Z},A,S)$ as a metric graph. 
Note that $\mathbb{Z} \times \mathbb{Z}$ can be quasi-isometrically embedded on $\hat{\Gamma}(F_2\times \mathbb{Z},A,S)$, and hence  $(F_2 \times \mathbb{Z},A)$ is not a relatively hyperbolic group. 
In \cite{Tp}, it is shown that $\hat{\Gamma}(F_2\times \mathbb{Z},A,S)$ admit a geodesic coarsely convex bicombing. 
Therefore, $(F_2 \times \mathbb{Z}, A)$ is a geodesic coarsely convex group pair. 

\subsection{Remark on the definition of coarsely convex group pair}

\begin{proposition}
\label{prop:extra-condition}
        If a group pair $(G,\mathcal{H})$ is {geodesic coarsely convex}, 
        then 
        $G$ acts geometrically relative to $\mathcal{H}$ on a geodesic coarsely convex space $X$ such that
    \begin{enumerate}[label=$($\alph*$)$]
        \item for every $H\in\mathcal{H}$ the fixed point set $X^H$ 
        is bounded, 
        \item for each $x\in X$ the isotropy $G_x$ is either finite 
        or conjugate to some $H\in\mathcal{H}$, and
        \item for each $H\in \mathcal{H}$, there exist $x\in X$
        such that $G_x = H$.
    \end{enumerate}
\end{proposition}

This proposition is a consequence of the construction and the two remarks below.


\begin{definition}[$\spike(X,\{H_i\},\{x_i\})$]\label{def:Spike}
    Let $\{H_i\mid i\in I\}$ be a collection of subgroups of a group $A$, $X$ an
$A$-space, and $\{x_i\mid i\in I\} \subset X$ such that  $x_i\in X^{H_i}$ for each $i\in I$.  The $A$-space
$\spike(X,\{H_i\},\{x_i\})$ is defined as the mapping cylinder of the
$A$-map $\psi\colon \bigsqcup_i (A/H_i) \to X$ defined by $H_i\mapsto x_i$   where $\bigsqcup_i (A/H_i)$ is considered a discrete space. More specifically, $\spike(X,\{H_i\},\{x_i\})$ is the quotient space arising as  $\left( [0,1] \times \bigsqcup_i (A/H_i) \right)  \sqcup X$ module the relations $(\psi(aH_i),0)\sim a.x_i$ for each $i\in I$ and $a\in A$.  Let $\hat x_i$ denote the point with representative $(\psi(aH_i),0)$. Any point in the $A$-orbit of a point $\hat x_i$ is called a \emph{spike-point}. 

In
the case that $X$ is a metric $A$-space, $\spike(X,\{H_i\},\{x_i\})$ is
considered a metric $A$-space by regarding each of the segments between $a.x_i$ and $a.\hat x_i$ for $a\in A$ and $i\in I$ having length a fixed number $\ell>0$; in this case, we use the notation $\spike(X,\{H_i\},\{x_i\}, \ell)$.  Note that there is
a natural $A$-equivariant embedding $X\to  \spike(X,\{H_i\},\{x_i\})$   which is a
topological (metric) embedding.  

In this article, we only use the construction of $\spike(X,\{H_i\},\{x_i\})$ in the case that the collection $\{H_i\}$ consists of one or two subgroups, and it will appear again in Section~\ref{sec:const-eqiv-tree}.
\end{definition}


\begin{remark}
 \label{remark:conedoff-preserve-gcc}
 It is an observation that if $X$ admits a $G$-equivariant coarsely convex geodesic bicombing, the same holds for $\spike(X,\{H_i\},\{x_i\})$. 
 \end{remark}

 \begin{remark}
  If $G$ acts geometrically on $X$ relative to $\mathcal{H}$, $C$ is a subgroup of some $H\in \mathcal{H}$, and $x\in X^H$, then $G$ acts geometrically on  $\spike(X, C, x)$ relative to $\mathcal{H}$.   
 \end{remark}

\begin{figure}[h]
\begin{tikzpicture}
\draw(0,0)--++(12,0)--++(-1,-1.5)--++(-12,0)--cycle; 
\draw(2,-0.5)node[below]{$x$};
\draw(2,-0.5)--++(1,2.5);
\fill(3,2)circle(0.06);
\draw(3,2)node[above]{$\hat{x}$};
\draw(2,-0.5)--++(0,3);
\draw(2,-0.5)--++(-1,2.5);
\fill(2,-0.5)circle(0.06);
\draw(5,-0.5)node[below]{$y$};
\draw(5,-0.5)--++(1,2.5);
\draw(5,-0.5)--++(0,3);
\draw(5,-0.5)--++(-1,2.5);
\fill(5,-0.5)circle(0.06);
\draw(8,-0.5)node[below]{$z$};
\draw(8,-0.5)--++(1,2.5);
\draw(8,-0.5)--++(0,3);
\draw(8,-0.5)--++(-1,2.5);
\fill(8,-0.5)circle(0.06);
\end{tikzpicture}
\caption{
This figure illustrates the structure of the Spike space.
Let $x,y,z \in X^H$.
From $x$, there emanete $G_x/C$ edge, each terminating at a spike-point. 
The stabilizer of each spike-point is of the form $gC$ for some $g \in G_x$.
}
\label{conesp}
\end{figure}
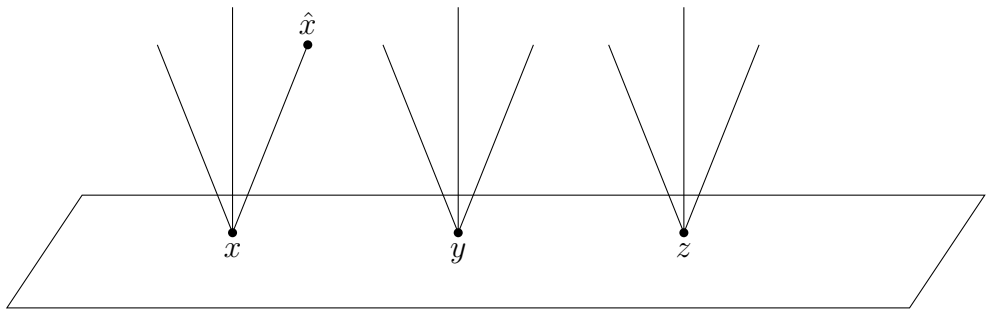

%% file: 25-Criteria-for-gcc.tex
\begin{definition}
    Let $Z$ be a metric space and let $\gamma\colon [0,1] \to Z$ be geodesic.
    Let $c\in [0,1]$. The \textit{reparametrization} of $\gamma|_{[0,c]}$ is 
    a geodesic $\gamma_c\colon [0,1]\to X$ defined by 
    $\gamma_c(t)= \gamma(ct)$.
\end{definition}

\begin{definition}
    Let $Z$ be a metric space and let $\gamma_1,\gamma_2\colon [0,1] \to Z$ be geodesic.
    Let $D\geq 0$ be a constant.
    We say that $\gamma_1$ and $\gamma_2$ coarsely $D$-fellow travel if 
    \[
    \max\{ d_Z(\gamma_1(0),\gamma_2(0)),d_Z(\gamma_1(1),\gamma_2(1)) \} \leq D
    \]
    and the Hausdorff distance between the images of $\gamma_1$ and $\gamma_2$ is at most $D$.
\end{definition}

\begin{definition}
\label{cftd}
A geodesic bicombing $\Gamma$ on a metric space $Z$ is \textit{coarsely $(C,D,E)$-forward thin}, if for any $v,w_1,w_2 \in Z$ such that $\min\{d_Z(v,w_1), d_Z(v,w_2)\}> 2D$,  the geodesics  $\gamma_1=\Gamma(v,w_1)$ and $\gamma_2=\Gamma(v,w_2)$ intersect a convex subspace $X$ such that $\Gamma$ restricted to $X\times X$ is a  geodesic $(E,C)$-coarsely convex bicombing on $X$ and the following conditions hold.

Let $p_1$ and $p_2$ be initial points of $\gamma_1$ and $\gamma_2$ in $X$ with $p_1=\gamma_1(c_0)$ and $p_2=\gamma_2(c_0')$.  Let $a$ and $b$ be the terminal points of $\gamma_1$ and $\gamma_2$ in $X$ with $a=\gamma_1(c_1)$ and $b=\gamma_2(c_1')$.
 
  \begin{enumerate}
         \item \label{Ex-3-2}
         The reparametrizations of $\gamma_1|_{[0,c_0]}$ and $\gamma_2|_{[0,c_0']}$ coarsely $D$-fellow travel. 
         \item 
         \label{Ex-3-3}
         $d_Z(a,b)\leq \max\{c_1, c_1'\}d_Z (\gamma_1(1), \gamma_2(1)) + 4D$ holds.
         \item 
         \label{Ex-3-4}
         $\gamma_1|_{[c_0,c_1]}=\Gamma(p_1,a)$ and $\gamma_2|_{[c_0',c_1']}=\Gamma(p_2,b)$ holds.
         \item 
         \label{Ex-3-5}
         $d_Z(\gamma_1(1),a)+d_Z(a,b)+d_Z(b,\gamma_2(1)) \leq d_Z(\gamma_1(1),\gamma_2(1)) +4D$ holds.
   \end{enumerate}
\end{definition}

\begin{definition}
\label{cbtd}
A geodesic bicombing $\Gamma$ on a metric space $Z$ is \textit{coarsely $(C,D,E)$-backward thin}, if for any $v_1,v_2,w \in Z$ such that $\min\{d_Z(v_1,w), d_Z(v_2,w)\}> 2D$,  the geodesics  $\gamma_1=\Gamma(v_1,w)$ and $\gamma_2=\Gamma(v_2,w)$ intersect a convex subspace $X$ such that $\Gamma$ restricted to $X\times X$ is a  geodesic $(E,C)$-coarsely convex bicombing on $X$ and the following conditions hold.

 Let $p_1$ and $p_2$ be initial points of $\gamma_1$ and $\gamma_2$ in $X$ with $p_1=\gamma_1(c_0)$ and $p_2=\gamma_2(c_0')$.
 Let $a$ and $b$ be the terminal points of $\gamma_1$ and $\gamma_2$ in $X$ with $a=\gamma_1(c_1)$ and $b=\gamma_2(c_1')$. 
     \begin{enumerate}
        \item 
        \label{Ey-3-2} The reparametrizations of 
         $\gamma_1|_{[c_1,1]})$ and $\gamma_2|_{[c_1',1]})$ coarsely $D$-fellow travel. 
         \item 
         \label{Ey-3-3}
         $d_Z(p_1,p_2)\leq \max\{1-c_0, 1-c_0'\}d_Z (\gamma_1(0), \gamma_2(0)) + 4D$ holds.
         \item 
         \label{Ey-3-4}
         $\gamma_1|_{[c_0,c_1]}=\Gamma(p_1,a)$ and $\gamma_2|_{[c_0',c_1']}=\Gamma(p_2,b)$ holds.
         \item 
         \label{Ey-3-5}
         $d_Z(\gamma_1(0),p_1)+d_Z(p_1,p_2)+d_Z(p_2,\gamma_2(0)) \leq d_Z(\gamma_1(0),\gamma_2(0)) +4D$ holds.
     \end{enumerate}
    
\end{definition}

\begin{theorem}
\label{consistent-and-thin-cconvex}
    Let $Z$ be a metric space with a geodesic bicombing $\Gamma$. 
    If $\Gamma$ is coarsely consistent, coarsely forward thin and coarsely backwards thin, then $\Gamma$ is  {geodesic coarsely convex}.
\end{theorem}

The theorem is an immediate consequence of Proposition~\ref{ccgccc} and the following two propositions. 

\begin{proposition}
\label{cft}
    Let $Z$ be a metric space with a geodesic bicombing $\Gamma$. 
    If $\Gamma$ is coarsely forward thin, then $\Gamma$ is  $(E,C)$-coarsely forward convex
    for some $E\geq 1$ and $C\geq 0$.
\end{proposition}

\begin{proposition}
\label{cbt}
    Let $Z$ be a metric space with a geodesic bicombing $\Gamma$. 
    If $\Gamma$ is coarsely backward thin, then $\Gamma$ is $(E,C)$-coarsely backward convex
    for some $E\geq 1$ and $C\geq 0$.
\end{proposition}

\begin{proof}[Proof of Proposition~\ref{cft}]
For $v,w_1,w_2 \in Z$, 
let $\gamma_1=\Gamma(v,w_1)$ and $\gamma_2=\Gamma(v,w_2)$.
Let $p_i$ be the initial point of $\gamma_i$ in $X$
and let $a$ be the terminal point of $\gamma_1$ in $X$
and let $b$ be the terminal point of $\gamma_2$ in $X$.
We set 
\[
p_1=\gamma_1(c_0), \quad
p_2=\gamma_2(c_0^{\prime}), \quad
a=\gamma_1(c_1), \quad
b=\gamma_2(c_1^{\prime}),
\]
where $c_0, c_1, c_0^{\prime}, c_1^{\prime} \in [0,1]$.
Without loss of generality, assume that $c_1 \leq c_1^{\prime}$.
Set $t=d_Z(\gamma_1(0),\gamma_1(1))$ and $s=d_Z(\gamma_2(0),\gamma_2(1))$.
We will show that 
there exist $E_{*} \geq 1$ and $C_* \geq 0$ depending only on $E$,$C$, and $D$ such that
for all $c \in [0,1]$, 
\begin{align*}
d_Z(\gamma_1(c), \gamma_2(c)) \leq cE_*d_Z(\gamma_1(1), \gamma_2(1))+C_*.
\end{align*}
We divide the proof into four cases based on the location of the parameter $c$ of $\gamma_1$ and $\gamma_2$: 
\begin{enumerate}
\renewcommand{\labelenumi}{\arabic{enumi}).}
\item \label{Exitem0} 
    $c \in [0,c_0]$ or $c \in [0,c_0']$.
\item \label{ExitemI} 
    $c \in [c_0,c_1]$ and $c \in [c_0',c_1']$ (see Figure \ref{FI}).
\item \label{ExitemII}
    $c \in [c_1,1]$ and $c \in [c_0',c_1']$ (see Figure \ref{FII}).
\item \label{ExitemIII}
    $c \in [c_1,1]$ and $c \in [c_1',1]$ (see Figure \ref{FIII}).
\end{enumerate} 
In Case \ref{Exitem0}), we suppose $c \in [0,c_0]$. 
By (\ref{Ex-3-2}) of Definition~\ref{cftd}, there exists $c' \leq c_0'$ such that
\begin{align*}
d_Z(\gamma_1(c),\gamma_2(c')) \leq 2D. 
\end{align*}
By the triangle inequality, we have
\begin{align*}
c's
&= d_Z(\gamma_2(0),\gamma_2(c')) \\
&\geq d_Z(\gamma_2(0),\gamma_1(c)) - d_Z(\gamma_1(c),\gamma_2(c')) \\
&\geq d_Z(\gamma_2(0),\gamma_1(c)) - 2D
=ct- 2D
\end{align*}
On the other hand, since $\gamma_2$ is a geodesic segment, 
\begin{align*}
d_Z(\gamma_2(c),\gamma_2(c'))=(c-c')s
\end{align*}
holds. 
Combing these inequalities, we have 
\begin{align*}
d_Z(\gamma_1(c),\gamma_2(c) 
&\leq d_Z(\gamma_1(c),\gamma_2(c')) +d_Z(\gamma_2(c'),\gamma_2(c)) \\
&\leq 2D + (c-c')s \\
&\leq cs-c's + 2D \\
&\leq cs-ct+2D +2D \\
&\leq c|t-s| +4D \\
&\leq cd_Z(\gamma_1(1),\gamma_2(1)) +4D. 
\end{align*}
\par 
In Case \ref{ExitemI}), by Definition~\ref{cftd}\eqref{Ex-3-4}, we have that
\[ \gamma_1(c)=\Gamma(p_1,a)\left( \dfrac{c-c_0}{c_1-c_0}\right), \quad \text{and} \quad 
\gamma_2(c)  =\Gamma(p_2,b)\left( \dfrac{c-c_0^{\prime}}{c_1^{\prime}-c_0^{\prime}}\right). \]
\begin{figure}[h]
\begin{tikzpicture}
\draw(0,0.5)--++(2,0.5)--++(0,-2)--++(-2,0.5)--cycle;
\draw(-3,0)--++(3,0.5);
\draw(-3,0)--++(3,-0.5);
\draw(-3,0)node[left]{$\gamma_1(0)$};
\draw(0,0.5)node[above]{$p_1$};
\draw(0,-0.5)node[below]{$p_2$};
\fill(-3,0)circle(0.06);
\draw(2,1)--++(1,1.5);
\draw(2,1)node[above]{$a$};
\draw(3,2.5)node[above]{$\gamma_1(1)$};
\fill(3,2.5)circle(0.06);
\draw(2,-1)--++(1,-1.5);
\draw(2,-1)node[right]{$b$};
\draw(3,-2.5)node[below]{$\gamma_2(1)$};
\fill(3,-2.5)circle(0.06);
\draw(1,0.75)node[above]{$\gamma_1(c)$};
\fill(1,0.75)circle(0.06);
\draw(1,-0.75)node[below]{$\gamma_2(c)$};
\fill(1,-0.75)circle(0.06);
\end{tikzpicture}
\caption{Proof of Proposition~\ref{cft}-Case \ref{ExitemI})}
\label{FI}
\end{figure}
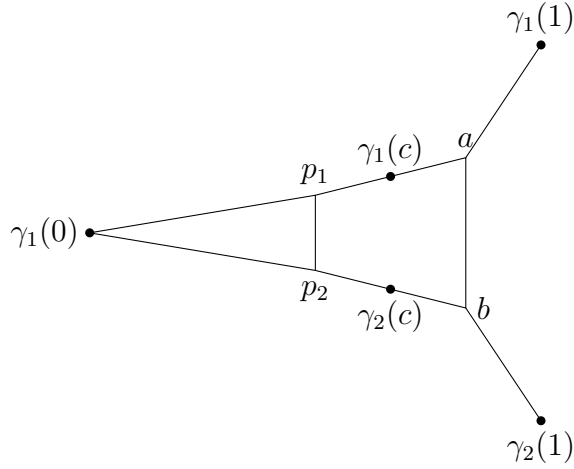
Note that
\[p_1 =\Gamma(p_1,a)(0), \quad p_2  =\Gamma(p_2,b)\left( 0 \right),\quad  a =\Gamma(p_1,a)\left(1\right), \quad \text{and}\quad  b  =\Gamma(p_2,b)\left(1\right) \]
We define $a^{\prime}$ to be
\begin{align*}
a^{\prime}\coloneqq \Gamma(p_1,a)\left( \dfrac{c-c_0^{\prime}}{c_1^{\prime}-c_0^{\prime}}\right).
\end{align*}
First, since $\Gamma$ is a geodesic $(E,C)$-coarsely convex bicombing, we have
\begin{align}
d_Z(a^{\prime}, \gamma_2(c)) 
&=d \left( \Gamma(p_1,a)\left( \dfrac{c-c_0^{\prime}}{c_1^{\prime}-c_0^{\prime}}\right) , \Gamma(p_2,b)\left( \dfrac{c-c_0^{\prime}}{c_1^{\prime}-c_0^{\prime}}\right) \right)  \notag \\
&\leq \dfrac{c-c_0^{\prime}}{c_1^{\prime}-c_0^{\prime}} \cdot Ed_Z(\Gamma(p_1,a)(1), \Gamma(p_2,b)(1))  \notag \\
&\ + \left( 1- \dfrac{c-c_0^{\prime}}{c_1^{\prime}-c_0^{\prime}} \right) \cdot Ed_Z(\Gamma(p_1,a)(0), \Gamma(p_2,b)(0)) +C\notag \\
&= \dfrac{c-c_0^{\prime}}{c_1^{\prime}-c_0^{\prime}} \cdot Ed_Z(a, b) +\left( 1- \dfrac{c-c_0^{\prime}}{c_1^{\prime}-c_0^{\prime}} \right)Ed_Z(p_1,p_2) +C\notag \\
&\leq \dfrac{cc_1^{\prime}-c_0^{\prime}c_1^{\prime}}{c_1^{\prime}-c_0^{\prime}} \cdot Ed_Z(\gamma_1(1), \gamma_2(1)) +4ED +2ED +C\label{C2-0} \\
&\leq c \cdot Ed_Z(\gamma_1(1), \gamma_2(1)) +6ED +C, \label{C2-1}
\end{align}
where \eqref{C2-0} follows from Definition~\ref{cftd} \eqref{Ex-3-2} and \eqref{Ex-3-3}, and
\eqref{C2-1} follows from $c_0^{\prime} \leq c \leq c_1^{\prime}$.
\par 
Next, we consider the distance between $a^{\prime}$ and $\gamma_1(c)$.
Then
\begin{align*}
d_Z(a^{\prime}, \gamma_1(c)) 
&=d \left( \Gamma(p_1,a)\left( \dfrac{c-c_0^{\prime}}{c_1^{\prime}-c_0^{\prime}}\right) , \Gamma(p_1,a)\left( \dfrac{c-c_0}{c_1-c_0}\right) \right)  \notag \\
&=\left| \dfrac{c-c_0^{\prime}}{c_1^{\prime}-c_0^{\prime}} -\dfrac{c-c_0}{c_1-c_0}\right| \cdot d_Z(p_1,a) \notag \\
&=\left| \dfrac{c-c_0^{\prime}}{c_1^{\prime}-c_0^{\prime}} -\dfrac{c-c_0}{c_1-c_0}\right| \cdot (c_1-c_0)t. \notag 
\end{align*}
Since $\gamma_1$ and $\gamma_2$ are geodesic segment, by the triangle inequality, we have
\begin{align*}
    c_0t-c_0's 
    &=d_Z(\gamma_1(0),\gamma_1(c_0t))-d_Z(\gamma_2(0),\gamma_2(c_0's)) \\
    &\leq d_Z(\gamma_1(c_0t),\gamma_2(c_0's)) \\
    &=d_Z(p_1,p_2). 
\end{align*}
By Item \eqref{Ex-3-2} of Definition~\ref{cftd}, we have $|c_0t-c_0's| \leq 2D$.
By the triangle inequality $ (c_1-c_0)t  \leq (c_1^{\prime}-c_0^{\prime})s+ d_Z(a,b)+ d_Z(p_1,p_2)$, we have
\begin{align}
\left( \dfrac{c-c_0^{\prime}}{c_1^{\prime}-c_0^{\prime}} -\dfrac{c-c_0}{c_1-c_0} \right)(c_1-c_0)t 
&=\dfrac{c-c_0^{\prime}}{c_1^{\prime}-c_0^{\prime}} (c_1-c_0)t  -\dfrac{c-c_0}{c_1-c_0}(c_1-c_0)t   \notag \\
&\leq \dfrac{c-c_0^{\prime}}{c_1^{\prime}-c_0^{\prime}}  \{ (c_1^{\prime}-c_0^{\prime})s+ d_Z(a,b) +d_Z(p_1,p_2)\} -(c-c_0)t   \notag \\
&\leq (c-c_0^{\prime})s -(c-c_0)t +\dfrac{c-c_0^{\prime}}{c_1^{\prime}-c_0^{\prime}} d_Z(a,b) +\dfrac{c-c_0^{\prime}}{c_1^{\prime}-c_0^{\prime}} d_Z(p_1,p_2) \notag\\
&\leq c(s-t) + (c_0t-c_0^{\prime}s) +cd_Z(\gamma_1(1),\gamma_2(1)) +4D +2D \label{E1-1}\\
&\leq c \cdot \{2d_Z(\gamma_1(1), \gamma_2(1))\} +8D. \label{E1-2}
\end{align}
where \eqref{E1-1} follows from Definition~\ref{cftd} \eqref{Ex-3-2} and \eqref{Ex-3-3}, and \eqref{E1-2} follows from $|c_0t-c_0's| \leq 2D$.
Similarly, by the triangle inequality $ (c_1^{\prime}-c_0^{\prime})s - d_Z(a,b)-d_Z(p_1,p_2)  \leq  (c_1-c_0)t$, we have
\begin{align}
\left( \dfrac{c-c_0}{c_1-c_0} - \dfrac{c-c_0^{\prime}}{c_1^{\prime}-c_0^{\prime}} \right) (c_1-c_0)t 
&=\dfrac{c-c_0}{c_1-c_0}(c_1-c_0)t  - \dfrac{c-c_0^{\prime}}{c_1^{\prime}-c_0^{\prime}} (c_1-c_0)t  \notag \\
&\leq (c-c_0)t -\dfrac{c-c_0^{\prime}}{c_1^{\prime}-c_0^{\prime}}  \{ (c_1^{\prime}-c_0^{\prime})s - d_Z(a,b)-d_Z(p_1,p_2)\} \notag \\
&= (c-c_0)t - (c-c_0^{\prime})s +\dfrac{c-c_0^{\prime}}{c_1^{\prime}-c_0^{\prime}} d_Z(a,b) +\dfrac{c-c_0^{\prime}}{c_1^{\prime}-c_0^{\prime}}d_Z(p_1,p_2) \notag \\
&\leq c(t-s) + (c_0's-c_0t) +cd_Z(\gamma_1(1),\gamma_2(1)) +4D +2D \label{E2-1}\\
&\leq c \cdot \{2d_Z(\gamma_1(1), \gamma_2(1))\} +8D. \label{E2-2}
\end{align}
Inequality \eqref{E2-1} follows from Item \eqref{Ex-3-2} and \eqref{Ex-3-3} of Definition~\ref{cftd}. 
Inequality \eqref{E2-2} follows from $|c_0t-c_0's| \leq 2D$.
Then we have
\begin{align}
d_Z(a^{\prime}, \gamma_1(c))
&=\left| \dfrac{c-c_0^{\prime}}{c_1^{\prime}-c_0^{\prime}} -\dfrac{c-c_0}{c_1-c_0}\right| \cdot (c_1-c_0)t \notag \\
&\leq c \cdot \{2d_Z(\gamma_1(1), \gamma_2(1))\} +8D. \label{C2-2}
\end{align}
Combining \eqref{C2-1} and \eqref{C2-2} yields
\begin{align*}
d_Z(\gamma_1(c),\gamma_2(c)) 
&\leq d_Z(\gamma_1(c),a^{\prime}) + d_Z(a^{\prime}, \gamma_2(c)) \\ 
& \leq c \cdot \{2d_Z(\gamma_1(1), \gamma_2(1))\} +8D + c \cdot Ed_Z(\gamma_1(1), \gamma_2(1)) +6ED +C \\
& = c(E+2) d_Z(\gamma_1(1), \gamma_2(1)) +6ED +8D +C.
\end{align*}
\par
In Case \ref{ExitemII}), we supposed that $c \in [c_1,1]$ and $c \in [c_0',c_1']$,
where $a=\gamma_1(c_1)$ and $b=\gamma_2(c^{\prime}_1)$. Let $
a^{\prime} = \gamma_1(cc_1)$ and , $b^{\prime} = \gamma_2(cc_1^{\prime})$. Applying the same argument as in Case \ref{ExitemI}), we obtain that
\begin{align}
\label{Q1}
d_Z(a^{\prime}, b^{\prime})
&\leq c(E+2) d_Z(a,b) +2ED +4D +C.
\end{align}
Hence, we have
\begin{align}
&d_Z(\gamma_1(c), \gamma_2(c)) \notag \\
& \leq d_Z(\gamma_1(c), a^{\prime}) + d_Z(a^{\prime},b^{\prime}) +d_Z(b^{\prime}, \gamma_2(c)) \notag \\
& = d_Z(\gamma_1(c), \gamma_1(cc_1) )+ d_Z(\gamma_1(cc_1),\gamma_2(cc_1'))+d_Z(\gamma_2(cc_1'), \gamma_2(c)) \notag \\
&\leq c(1-c_1)t + c(E+2) d_Z(a,b) +2ED +4D +C + c(1-c_1^{\prime})s \label{Q1-1}\\
&= cd_Z(\gamma_1(1),\gamma_1(c_1)) +c(E+2) d_Z(a,b) + 2ED + 4D +C + cd_Z(\gamma_2(c_1^{\prime}), \gamma_2(1)) \notag \\
&\leq c(E+2) \{ d_Z(\gamma_1(1), a) + d_Z(a,b) + d_Z(b, \gamma_2(1)) \} +2ED +4D +C \notag \\
&\leq c(E+2) \{d_Z(\gamma_1(1),\gamma_2(1)) +4D\} +2ED +4D  +C \label{Q1-2}\\
&\leq c(E+2) d_Z(\gamma_1(1),\gamma_2(1)) +4D(E+2) +2ED +4D  +C \notag \\
&= c(E+2) d_Z(\gamma_1(1),\gamma_2(1)) +6ED+12D  +C, \notag 
\end{align}
where \eqref{Q1-1} follows from \eqref{Q1}, and 
\eqref{Q1-2} follows from Definition~\ref{cftd} \eqref{Ex-3-5}.
\begin{figure}[t]
\begin{tikzpicture}
\draw(0,0.5)--++(2,0.5)--++(0,-2)--++(-2,0.5)--cycle;
\draw(-3,0)--++(3,0.5);
\draw(-3,0)--++(3,-0.5);
\draw(-3,0)node[left]{$\gamma_1(0)$};
\draw(0,0.5)node[above]{$p_1$};
\draw(0,-0.5)node[below]{$p_2$};
\fill(-3,0)circle(0.06);
\draw(2,1)--++(1,1.5);
\draw(2,1)node[above]{$a$};
\draw(3,2.5)node[above]{$\gamma_1(1)$};
\fill(3,2.5)circle(0.06);
\draw(2,-1)--++(1,-1.5);
\draw(2,-1)node[right]{$b$};
\draw(2.5,-1.75)node[right]{$\gamma_2(c)$};
\fill(2.5,-1.75)circle(0.06);;
\draw(3,-2.5)node[below]{$\gamma_2(1)$};
\fill(3,-2.5)circle(0.06);
\draw(0.5,0.625)node[below]{$a'$};
\fill(0.5,0.625)circle(0.06);
\draw(1,0.75)node[above]{$\gamma_1(c)$};
\fill(1,0.75)circle(0.06);
\draw(1,-0.75)node[below]{$b'$};
\fill(1,-0.75)circle(0.06);
\end{tikzpicture}
\caption{Proof of Proposition~\ref{cft}-Case \ref{ExitemII})}
\label{FII}
\end{figure}
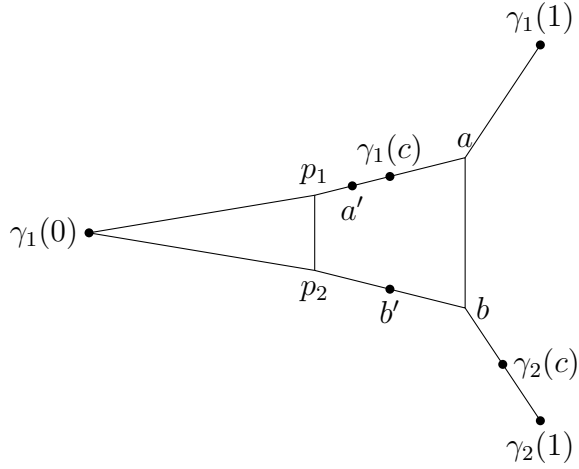

Case \ref{ExitemIII}). Supposed that $c \in [c_1,1]$ and $c \in [c_1',1]$.
Note that $c_1 \leq c^{\prime}_1 \leq c$.
Then, we have
\begin{align}
&d_Z(\gamma_1(c),\gamma_2(c))  \notag \\
&\leq d_Z(\gamma_1(c),a) + d_Z(a,b) + d_Z(b, \gamma_2(c))  \notag \\
&=d_Z(\gamma_1(c),\gamma_1(c_1)) + d_Z(a,b)+ d_Z(\gamma_2(c^{\prime}_1),\gamma_2(c))  \notag\\
&=c\left(1-\dfrac{c_1}{c} \right)t + c^{\prime}_1d_Z(\gamma_1(1),\gamma_2(1)) +4D + c\left( 1-\dfrac{c^{\prime}_1}{c} \right)s  \label{P1-1}\\
&\leq c(1-c_1)t + cd_Z(\gamma_1(1),\gamma_2(1))+ c(1-c^{\prime}_1)s +4D \notag \\
&= c \{ d_Z(\gamma_1(1),a) + d_Z(b,\gamma_2(1)) \} + cd_Z(\gamma_1(1),\gamma_2(1)) +4D \notag \\
&\leq c \{ d_Z(\gamma_1(1),a) + d_Z(a,b)+d_Z(b,\gamma_2(1)) \} + cd_Z(\gamma_1(1),\gamma_2(1)) +4D \notag\\
&\leq c \{ d_Z(\gamma_1(1),\gamma_2(1)) +4D\} + cd_Z(\gamma_1(1),\gamma_2(1)) +4D \label{P1-2}\\
&\leq c \{2 d_Z(\gamma_1(1),\gamma_2(1))\} +8D. \notag
\end{align}
where \eqref{P1-1} follows from Definition~\ref{cftd} \eqref{Ex-3-3}, and \eqref{P1-2} follows from Definition~\ref{cftd} \eqref{Ex-3-5}.
\par
Therefore, $\gamma_1$ and $\gamma_2$ satisfy that, for any $c \in [0,1]$, 
\begin{align*}
    d_Z(\gamma_1(c), \gamma_2(c)) 
    \leq c(E+2) d_Z(\gamma_1(1),\gamma_2(1)) +6ED+12D  +C 
\end{align*}
This completes the proof. 
\begin{figure}[h]
\begin{tikzpicture}
\draw(0,0.5)--++(2,0.5)--++(0,-2)--++(-2,0.5)--cycle;
\draw(-3,0)--++(3,0.5);
\draw(-3,0)--++(3,-0.5);
\draw(-3,0)node[left]{$\gamma_1(0)$};
\draw(0,0.5)node[above]{$p_1$};
\draw(0,-0.5)node[below]{$p_2$};
\fill(-3,0)circle(0.06);
\draw(2,1)--++(1,1.5);
\draw(2,1)node[above]{$a$};
\draw(3,2.5)node[above]{$\gamma_1(1)$};
\fill(3,2.5)circle(0.06);
\draw(2,-1)--++(1,-1.5);
\draw(2,-1)node[right]{$b$};
\fill(2.5,-1.75)circle(0.06);;
\draw(3,-2.5)node[below]{$\gamma_2(1)$};
\fill(3,-2.5)circle(0.06);
\draw(2.5,1.5)node[right]{$\gamma_1(c)$};
\fill(2.5,1.75)circle(0.06);
\draw(3,-1.5)node{$\gamma_2(c)$};
\fill(2.5,-1.75)circle(0.06);
\end{tikzpicture}
\caption{Proof of Proposition~\ref{cft}-Case \ref{ExitemIII})}
\label{FIII}
\end{figure}
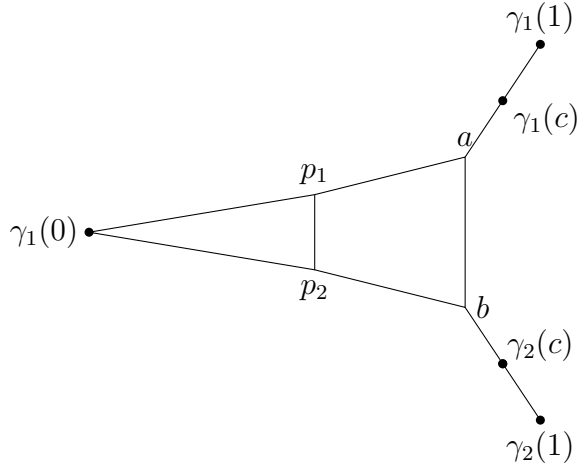
\end{proof}

\begin{proof}[Proof of Proposition~\ref{cbt}]
For $v,w_1,w_2 \in Z$, 
let $\gamma_1=\Gamma(v,w_1)$ and $\gamma_2=\Gamma(v,w_2)$.
Let $p_i$ be the initial point of $\gamma_i$ in $X$
and let $a$ be the terminal point of $\gamma_1$ in $X$
and let $b$ be the terminal point of $\gamma_2$ in $X$.
We set 
\[
p_1=\gamma_1(c_0), \quad
p_2=\gamma_2(c_0^{\prime}), \quad
a=\gamma_1(c_1), \quad
b=\gamma_2(c_1^{\prime}),
\]
where $c_0, c_1, c_0^{\prime}, c_1^{\prime} \in [0,1]$.
Without loss of generality, assume that $1-c_0 \leq 1-c_0^{\prime}$.
Set $t=d_Z(\gamma_1(0),\gamma_1(1))$ and $s=d_Z(\gamma_2(0),\gamma_2(1))$.
We will show that 
there exist $E_{*} \geq 1$ and $C_* \geq 0$ depending only on $E$, $C$, and $D$ such that
for all $c \in [0,1]$, 
\begin{align*}
d_Z(\gamma_1(c), \gamma_2(c)) \leq (1-c)E_*d_Z(\gamma_1(0), \gamma_2(0))+C_*.
\end{align*}
We divide the proof into the following cases based on the location of the parameter $c$ of $\gamma_1$ and $\gamma_2$: 
\begin{enumerate}
\renewcommand{\labelenumi}{\arabic{enumi}).}
\item \label{Eyitem0} 
    $c \in [c_1,1]$ or $c \in [c_1',1]$.
\item \label{EyitemI} 
    $c \in [c_0,c_1]$ and $c \in [c_0',c_1']$ (see Figure \ref{FyI}).
\item \label{EyitemII}
    $c \in [0,c_0]$ and $c \in [c_0',c_1']$ (see Figure \ref{FyII}).
\item \label{EyitemIII}
    $c \in [0,c_0]$ and $c \in [0,c_0']$ 
\end{enumerate} 
In Case \ref{Eyitem0}). we suppose $c \in [c_1,1]$. 
By Item \eqref{Ey-3-2} of Definition~\ref{cbtd}, there exists $c' \leq c_0'$ such that
\begin{align*}
d_Z(\gamma_1(c),\gamma_2(c')) \leq 2D. 
\end{align*}
By the triangle inequality, we have
\begin{align*}
c's
&= d_Z(\gamma_2(0),\gamma_2(c')) \\
&\geq d_Z(\gamma_2(0),\gamma_1(c)) - d_Z(\gamma_1(c),\gamma_2(c')) \\
&\geq d_Z(\gamma_2(0),\gamma_1(c)) - 2D
=ct- 2D
\end{align*}
On the other hand, since $\gamma_2$ is a geodesic segment, 
\begin{align*}
d_Z(\gamma_2(c),\gamma_2(c'))=(c-c')s
\end{align*}
holds. 
Combing these inequalities, we have 
\begin{align*}
d_Z(\gamma_1(c),\gamma_2(c) 
&\leq d_Z(\gamma_1(c),\gamma_2(c')) +d_Z(\gamma_2(c'),\gamma_2(c)) \\
&\leq 2D + (c-c')s \\
&\leq cs-c's + 2D \\
&\leq cs-ct+2D +2D \\
&\leq c|t-s| +4D \\
&\leq cd_Z(\gamma_1(1),\gamma_2(1)) +4D. 
\end{align*}
\par 
In Case \ref{EyitemI}).  Definition~\ref{cbtd}\eqref{Ey-3-4} implies that
\[
\gamma_1(c) =\Gamma(p_1,a)\left( \dfrac{c-c_0}{c_1-c_0}\right) \quad \text{and} \quad 
\gamma_2(c)  =\Gamma(p_2,a)\left( \dfrac{c-c_0^{\prime}}{c_1^{\prime}-c_0^{\prime}}\right).
\]
Note that
\[ p_1=\Gamma(p_1,a)(0), \quad p_2 =\Gamma(p_2,b)\left( 0 \right), \quad  a =\Gamma(p_1,a)\left(1\right), \quad  \text{and} \quad b =\Gamma(p_2,b)\left(1\right).\]
Define $a^{\prime}$ to be
\begin{align*}
a^{\prime}\coloneqq \Gamma(p_1,a)\left( \dfrac{c-c_0^{\prime}}{c_1^{\prime}-c_0^{\prime}}\right).
\end{align*}
First, since $\Gamma$ is a geodesic $(E,C)$-coarsely convex bicombing, we have
\begin{align}
d_Z(a^{\prime}, \gamma_2(c)) 
&=d_Z \left(\Gamma(p_1,a)\left( \dfrac{c-c_0^{\prime}}{c_1^{\prime}-c_0^{\prime}}\right) , \Gamma(p_2,a)\left( \dfrac{c-c_0^{\prime}}{c_1^{\prime}-c_0^{\prime}}\right) \right)  \notag \\
&\leq \dfrac{c-c_0^{\prime}}{c_1^{\prime}-c_0^{\prime}} \cdot Ed_Z(\Gamma(p_1,a)(1),\Gamma(p_2,b)(1)) \notag \\
&\ +\left(1-\dfrac{c-c_0'}{c_1^{\prime}-c_0^{\prime}}\right) \cdot Ed_Z(\Gamma(p_1,a)(0),\Gamma(p_2,b)(0)) +C\notag \\
&=\dfrac{c-c_0^{\prime}}{c_1^{\prime}-c_0^{\prime}} \cdot Ed_Z(a,b) +\left(1-\dfrac{c-c_0'}{c_1^{\prime}-c_0^{\prime}}\right) \cdot Ed_Z(p_1,p_2) +C\notag \\
&\leq 2ED + \dfrac{(1-c)(1-c_0')-(1-c_1')(1-c_0')}{(1-c_0')-(1-c_1')} \cdot Ed_Z(\gamma_1(0), \gamma_2(0)) +4ED +C\label{C2y-0} \\
&\leq (1-c) \cdot Ed_Z(\gamma_1(0), \gamma_2(0)) +6ED +C, \label{C2y-1}
\end{align}
where \eqref{C2y-0} follows from 
Definition~\ref{cbtd} \eqref{Ey-3-2} and \eqref{Ey-3-3}, and 
\eqref{C2y-1} follows from $1-c_0 \leq 1-c_0'$.
\par
On the other hand, $d_Z(a^{\prime},\gamma_1(c))$ can be estimated as follows.
First, since $\gamma_1$ and $\gamma_2$ are geodesic segment, by the triangle inequality, we have
\begin{align*}
    |(1-c_1)t-(1-c_1')s|
    &= |d_Z(\gamma_1(c_1t),\gamma_1(1))-d_Z(\gamma_2(c_1's),\gamma_2(1))| \\
    &\leq d_Z(\gamma_1(c_1t),\gamma_2(c_1's)) \\
    &=d_Z(a,b). 
\end{align*}
By Definition~\ref{cbtd} \eqref{Ey-3-2}, we have $|(1-c_1)t-(1-c_1')s| \leq 2D$.
\begin{align*}
d_Z(a^{\prime}, \gamma_1(c)) 
&=d_Z \left( \Gamma(p_1,a)\left( \dfrac{c-c_0^{\prime}}{c_1^{\prime}-c_0^{\prime}}\right) , \Gamma(p_1,a)\left( \dfrac{c-c_0}{c_1-c_0}\right) \right)  \notag \\
&=\left| \dfrac{c-c_0^{\prime}}{c_1^{\prime}-c_0^{\prime}} -\dfrac{c-c_0}{c_1-c_0}\right| \cdot d_Z(p_1,a) \notag \\
&=\left| \dfrac{c-c_0^{\prime}}{c_1^{\prime}-c_0^{\prime}} -\dfrac{c-c_0}{c_1-c_0}\right| \cdot (c_1-c_0)t. \notag 
\end{align*}
Since $(c_1-c_0)t =d_Z(p_1,a)$ and $(c_1'-c_0')s=d_Z(p_2,a)$,  
the triangle inequality implies $ (c_1-c_0)t  \geq (c_1^{\prime}-c_0^{\prime})s - d_Z(a,b) -d_Z(p_1,p_2)$. 
Therefore, 
\begin{align}
\left( \dfrac{c-c_0^{\prime}}{c_1^{\prime}-c_0^{\prime}} -\dfrac{c-c_0}{c_1-c_0} \right) (c_1-c_0)t 
&=\left\{ \left(1-\dfrac{c-c_0}{c_1-c_0}\right) - \left(1-\dfrac{c-c_0^{\prime}}{c_1^{\prime}-c_0^{\prime}}\right)  \right\} (c_1-c_0)t  \notag \\
&=\left( \dfrac{c_1-c}{c_1-c_0} - \dfrac{c_1'-c}{c_1^{\prime}-c_0^{\prime}} \right) (c_1-c_0)t  \notag \\
&= (c_1-c)t- \dfrac{c_1'-c}{c_1^{\prime}-c_0^{\prime}} (c_1-c_0)t \notag \\
&\leq (c_1-c)t-  \dfrac{c_1'-c}{c_1^{\prime}-c_0^{\prime}}  \{ (c_1^{\prime}-c_0^{\prime})s -d_Z(a,b) - d_Z(p_1,p_2)\} \notag \\
&= (c_1-c)t-  (c_1'-c)s + \dfrac{c_1'-c}{c_1'-c_0'} d_Z(a,b)+ \dfrac{c_1'-c}{c_1'-c_0'} d_Z(p_1,p_2). \label{Yy-1}
\end{align}
By inequality $|(1-c_1)t-(1-c_1')s| \leq 2D$, 
the sum of the first two terms of \eqref{Yy-1} can be estimated as follows: 
\begin{align*}
(c_1-c)t-  (c_1'-c)s 
&=\{(1-c)-(1-c_1)\}t-  \{(1-c)-(1-c_1')\}s \\
&=(1-c)(t-s) + \{(1-c_1')s-(1-c_1)t\}\\
&\leq (1-c) d_Z(\gamma_1(0), \gamma_2(0)) + 2D.
\end{align*}
By Definition~\ref{cbtd} \eqref{Ey-3-2}, the third term of \eqref{Yy-1} is less than or equal to $2D$. 
By Definition~\ref{cbtd} \eqref{Ey-3-3}, we estimate the fourth term of \eqref{Yy-1}
\begin{align*}
\dfrac{c_1'-c}{c_1^{\prime}-c_0^{\prime}} d_Z(p_1, p_2) 
&= \dfrac{(1-c)-(1-c_1')}{(1-c_0')-(1-c_1')} d_Z(p_1, p_2) \\
&\leq \dfrac{(1-c)(1-c_0')-(1-c_1')(1-c_0')}{(1-c_0')-(1-c_1')} d_Z(\gamma_1(0), \gamma_2(0)) +4D\\
&\leq (1-c) d_Z(\gamma_1(0), \gamma_2(0)) +4D. 
\end{align*}
Thus, we obtain
\begin{align*}
\left( \dfrac{c-c_0^{\prime}}{c_1^{\prime}-c_0^{\prime}} -\dfrac{c-c_0}{c_1-c_0} \right) (c_1-c_0)t \leq (1-c) \{2d_Z(\gamma_1(0), \gamma_2(0))\} +8D. 
\end{align*}
By a similar argument, we have
\begin{align*}
\left( \dfrac{c-c_0}{c_1-c_0} -\dfrac{c-c_0'}{c_1'-c_0'} \right) (c_1-c_0)t \leq (1-c) \{2d_Z(\gamma_1(0), \gamma_2(0))\}+8D. 
\end{align*}
Therefore, 
\begin{align}
d_Z(a^{\prime}, \gamma_1(c))
&=\left| \dfrac{c-c_0^{\prime}}{c_1^{\prime}-c_0^{\prime}} -\dfrac{c-c_0}{c_1-c_0}\right| \cdot (c_1-c_0)t \notag \\
&\leq (1-c) \cdot \{2d_Z(\gamma_1(0), \gamma_2(0))\} +8D. \label{C2y-2}
\end{align}
holds.
Combining \eqref{C2y-1} and \eqref{C2y-2} yields
\begin{align*}
d_Z(\gamma_1(c),\gamma_2(c)) 
&\leq d_Z(\gamma_1(c),a^{\prime}) + d_Z(a^{\prime}, \gamma_2(c)) \\ 
& \leq (1-c) \cdot \{2d_Z(\gamma_1(1), \gamma_2(1))\} +8D + (1-c) \cdot Ed_Z(\gamma_1(1), \gamma_2(1)) +6ED +C \\
& = (1-c) (E+2) d_Z(\gamma_1(1), \gamma_2(1)) +6ED +8D +C.
\end{align*}
 
\begin{figure}[t]
\begin{tikzpicture}
\draw(0,1)--++(2,-0.5)--++(0,-1)--++(-2,-0.5)--cycle;
\draw(5,0)--++(-3,0.5);
\draw(5,0)--++(-3,-0.5);
\draw(2,-0.5)node[below]{$b$};
\draw(-2,2.5)--++(2,-1.5);
\draw(-2,2.5)node[left]{$\gamma_1(0)$};
\fill(-2,2.5)circle(0.06);
\draw(0,1)node[left]{$p_1$};
\draw(0,-1)node[left]{$p_2$};
\draw(2,0.5)node[above]{$a$};
\draw(5,0)node[above]{$\gamma_1(1)$};
\fill(5,0)circle(0.06);
\draw(-2,-2.5)--++(2,1.5);
\draw(-2,-2.5)node[below]{$\gamma_2(0)$};
\fill(-2,-2.5)circle(0.06);
\draw(1,0.75)node[above]{$\gamma_1(c)$};
\fill(1,0.75)circle(0.06);
\draw(1,-1.25)node{$\gamma_2(c)$};
\fill(1,-0.75)circle(0.06);
\end{tikzpicture}
\caption{Proof of Proposition~\ref{cbt}-Case \ref{EyitemI}}
\label{FyI}
\end{figure}
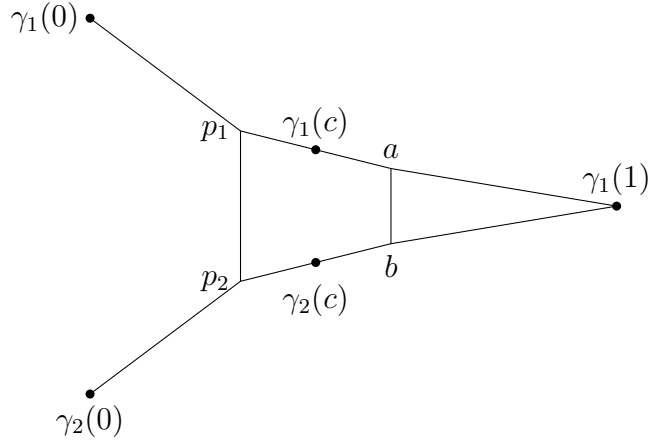
\par
In Case \ref{EyitemII}). Supposed that $c \in [0,c_0]$ and $c \in [c_0', c_1']$,
where $p_1=\gamma_1(c_0)$ and $p_2=\gamma_2(c^{\prime}_0)$.
Set $a^{\prime} = \gamma_1(c_0+(1-c_0)c))$ and $b'=\gamma'_2(c'_0+(1-c'_0)c)$. 
We apply the same argument as in Case \ref{EyitemI}) to $a'=\gamma_1(c_0+(1-c_0)c))$ and $b'=\gamma_2(c'_0+(1-c'_0)c)$, we have
\begin{align}
\label{Ay}
d_Z(a^{\prime}, b')
\leq (1-c)(E+2) d_Z(p_1,p_2) +2ED +4D +C .
\end{align}
Set $K_1=2ED +4D +C$. 
Note that 
\[ d_Z(\gamma_1(c), a^{\prime}) =(1-c)c_0t \quad \text{and}\quad  
 d_Z(b^{\prime},\gamma_2(c)) =(1-c)c'_0s  \]
The distance $d_Z(\gamma_1(c), \gamma_2(c))$ can be estimated as follows:
\begin{align}
&d_Z(\gamma_1(c), \gamma_2(c)) \notag \\
& \leq d_Z(\gamma_1(c), a^{\prime}) + d_Z(a',b')+ d_Z(b^{\prime},\gamma_2(c)) \notag \\
&= (1-c)c_0t+ (1-c)c'_0s+(1-c)(E+2) d_Z(p_1, p_2) + K_1 \label{Ay-1} \\
&= (1-c)\{d_Z(\gamma_1(0),p_1) +d_Z(\gamma_2(0),p_2)\}+(1-c)(E+2) d_Z(p_1, p_2) +K_1 \notag \\
&= (1-c)(E+2)\{d_Z(\gamma_1(0),p_1) +d_Z(p_1, p_2) +d_Z(\gamma_2(0),p_2)\}+K_1 \notag \\
&\leq (1-c)(E+2)d_Z(\gamma_1(0),\gamma_2(0)) +4D(E+2)+ K_1 \label{Ay-2},
\end{align}
where \eqref{Ay-1} follows from \eqref{Ay}, and
\eqref{Ay-2} follows Definition~\ref{cbtd} \eqref{Ey-3-5}.
\begin{figure}[t]
\begin{tikzpicture}
\draw(0,1)--++(2,-0.5)--++(0,-1)--++(-2,-0.5)--cycle;
\draw(5,0)--++(-3,0.5);
\draw(5,0)--++(-3,-0.5);
\draw(-1.25,1.75)node[below]{$\gamma_1(c)$};
\fill(-1,1.75)circle(0.06);
\draw(2,-0.5)node[below]{$b$};
\draw(-2,2.5)--++(2,-1.5);
\draw(-2,2.5)node[left]{$\gamma_1(0)$};
\fill(-2,2.5)circle(0.06);
\draw(0,1)node[left]{$p_1$};
\draw(0,-1)node[left]{$p_2$};
\draw(2,0.5)node[above]{$a$};
\draw(5,0)node[above]{$\gamma_1(1)$};
\fill(5,0)circle(0.06);
\draw(-2,-2.5)--++(2,1.5);
\draw(-2,-2.5)node[below]{$\gamma_2(0)$};
\fill(-2,-2.5)circle(0.06);
\draw(1,0.75)node[above]{$a'$};
\fill(1,0.75)circle(0.06);
\draw(1,-1.25)node{$\gamma_2(c)$};
\fill(1,-0.75)circle(0.06);
\fill(1.5,-0.63)circle(0.06);
\draw(1.5,-0.63)node[below]{$b'$};
\end{tikzpicture}
\caption{Proof of Proposition~\ref{cbt}-Case \ref{EyitemII}) 
}
\label{FyII}
\end{figure}
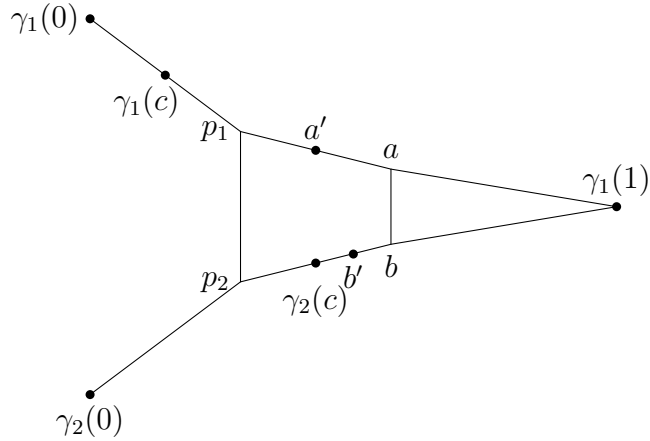

In Case \ref{EyitemIII}). Suppose  that $c \in [0,c_0]$ and $c \in [0, c_0']$, where $p_1=\gamma_1(c_0)$ and $p_2=\gamma_2(c^{\prime}_0)$.
In the following sequences of inequalities, 
the first one follows from Item \eqref{Ey-3-3} of Definition~\ref{cbtd}, 
the second one from $0< 1-c\leq 1$, 
the third one from $1-c^{\prime}_0 \leq 1-c$, 
the last one from Item \eqref{Ey-3-5} of Definition~\ref{cbtd}.
\begin{align*}
&d_Z(\gamma_1(c),\gamma_2(c))  \\
&=d_Z(\gamma_1(c),\gamma_1(c_0)) + d_Z(p_1,p_2)+ d_Z(\gamma_2(c^{\prime}_0),\gamma_2(c)) \\
&\leq (c_0-c)t + (1-c_0')d_Z(\gamma_1(0),\gamma_2(0)) +4D + (c_0'-c)s \\
&=\{ (1-c)-(1-c_0) \}t + \{(1-c)-(1-c_0')\}s + (1-c_0')d_Z(\gamma_1(0),\gamma_2(0)) +4D\\
&=(1-c)\left(1-\dfrac{1-c_0}{1-c} \right)t + (1-c)\left( 1-\dfrac{1-c^{\prime}_0}{1-c} \right)s + (1-c_0')d_Z(\gamma_1(0),\gamma_2(0)) +4D\\
&\leq (1-c)\{1-(1-c_0) \}t + (1-c)\{ 1-(1-c^{\prime}_0) \}s + (1-c_0')d_Z(\gamma_1(0),\gamma_2(0)) +4D\\
&=(1-c)c_0t + (1-c)c^{\prime}_0s + (1-c_0')d_Z(\gamma_1(0),\gamma_2(0)) +4D\\
&\leq (1-c)d_Z(\gamma_1(0),p_1) + (1-c)d_Z(\gamma_2(0),p_2) + (1-c)d_Z(\gamma_1(0),\gamma_2(0)) +4D\\
&\leq (1-c) \{2 d_Z(\gamma_1(0),\gamma_2(0)) \}+8D.
\end{align*}
Therefore, 
\begin{align*}
d_Z(\gamma_1(c), \gamma_2(c)) \leq (1-c)(2+E) d_Z(\gamma_1(0),\gamma_2(0)) + C
\end{align*}
holds for all $c \in [0,1]$.
This shows that $\Gamma$ is $(2+E,C)$-coarsely backward convex. 
\end{proof}

%% file: 30-SpacesAmalgamatedProducts.tex
\label{sec:const-eqiv-tree}
A \emph{tree of  spaces} is a continuous map 
$\xi\colon Z\to T$ where $Z$ is a geodesic metric space, $T$ is a bipartite tree with $V(T)=K\sqcup L$, such that  
 \begin{enumerate}
  \item for each $l\in L$, the inverse $\xi^{-1}(l)$ consists of a
        single point, and
  \item for each $k\in K$, the inverse $X_k:=\xi^{-1}(\starNb(k))$
        is a convex subspace of $Z$. 
 \end{enumerate}
 The space $X_k$ is called \emph{the vertex space at $k$} for $k\in K$. 
 
A \emph{$G$-tree of spaces} is a tree of spaces $\xi\colon Z \to T$ that
is also $G$-map for the group $G$, the $G$-action on $Z$ is by
isometries, and the $G$-action on $T$ is cellular and preserves the
blocks of the bipartition $V(T)=K\sqcup L$.

In this section we construct a $G$-tree-spaces $\xi\colon Z \to T$ from a group splitting of $G$ and spaces corresponding to vertex groups of the splitting. There are two analogous results, one for amalgamated products and another for HNN-extensions. The vertex-spaces of the tree of spaces $\xi$  are coned-off spaces that are defined below. Most of the section is an adaptation of methods in~\cite{BiMa23}. 
 
\begin{theorem}\label{thm:CombinationFine}
For $i\in\{
 1,2\}$, let $(G_i, \mathcal{H}_i\cup\{K_i\})$ be a pair and   $\partial_i\colon C\to K_i$  a group monomorphism. Let $G=G_1\ast_C G_2$  denote the amalgamated product determined by
$G_1\xleftarrow{\partial_1}C\xrightarrow{\partial_2}G_2$, and  $\mathcal{H}=\mathcal{H}_1\cup\mathcal{H}_2$.
Let $X_i$ be a $G_i$-space  that has  a point $x_i$ with $G_i$-stabilizer $K_i$. Then there is a $G$-space $Z$ with the following properties:
\begin{enumerate}[series=amalgamation]
    \item \label{item:Amalgamation-basepoint}
          $Z$ has a point $z$ such that the $G$-stabilizer  $G_z=\langle K_1,K_2\rangle$. 
    
    \item \label{item:Amalgamation-topological-embedding-CF}
          There is a $G_i$-equivariant  topological embedding
    of coned-off spaces 
    $\spike(X_i, C, x_i)\hookrightarrow Z$ that maps $\hat x_i$ to $z$.

    \item \label{item:Amalgamation-cocompact}
       If the $G_i$-action on $X_i$ is  cocompact  for $i=1,2$, 
       then the $G$-action on  $Z$ is cocompact.
        
     \item \label{item:Amalgamation-connected}
           If $X_i$ is connected  for $i=1,2$, then $Z$ is connected.

     \item \label{item:Amalgamation-H-is-stabilizer}
           If  $H\in\mathcal{H}_i\cup\{K_i\}$ is the $G_i$-stabilizer of a point of $X_i$ for $i=1,2$, then every $H\in \mathcal{H}\cup\{\langle K_1,K_2 \rangle\}$ is the $G$-stabilizer of a point of $Z$.
   
     \item \label{item:Amalgamation-Stabilizers-are-conjugate}
           If point $G_i$-stabilizers in $X_i$ are finite or conjugates of subgroups in $\mathcal{H}_i\cup\{ K_i\}$ for $i=1,2$, then point $G$-stabilizers in $Z$ are finite or conjugates of subgroups in $\mathcal H\cup\{ \langle K_1,K_2\rangle\}$.

     \item \label{item:Amalgamation-tree}
           There is a  $G$-tree $T$ and a $G$-map $\xi\colon  Z \to T$ of  with the following properties:
  \begin{enumerate}
      \item $\xi (z)$ is a vertex of $T$ and the preimage $\xi^{-1}(\xi (z))$ consists of only the point $z$.
      
      \item The image $\xi (X_i)$ is a single vertex $v_i$ of $T$. 
      
      \item The preimage $\xi^{-1}(\mathsf{Star}(v_i))$ is homeomorphic to $\spike(X_i, C, x_i)$.
     
      \item The $G$-orbit of $\xi(z)$ and its complement in the set of vertices of $T$ make $T$ a bipartite graph.
  \end{enumerate}
\end{enumerate}
Moreover, 
\begin{enumerate}[resume=amalgamation]
 \item \label{item:Amalgamation-geometric-embedding-CF}
       If $X_i$ is a geodesic metric space and the $G_i$-action on $X_i$
       is by isometries, then there is $\ell>0$ such that $Z$ is a
       geodesic metric space, the $G$-action is by isometries, and the
       embeddings $\spike(X_i, C,x_i, \ell)\hookrightarrow Z$
       are isometric and convex.

 \item \label{item:Amalgamation-relative-p.d} 
       If $X_i$ is a geodesic metric space 
       and the $G_i$-action on $X_i$ is properly
       discontinuous relative to $\mathcal{H}_i\cup\{K_i\}$ for $i=1,2$,
       then the $G$-action on $Z$ is properly discontinuous relative to
       $\mathcal{H}\cup \{\langle K_1,K_2 \rangle\}$.

 \item \label{item:Amalgamation-Cayley-Abels-graph}
       If $X_i$ is a Cayley-Abels graph for $(G_i,\mathcal{H}_i\cup
       K_i)$, then one can assume that $\xi^{-1}(v_i)$ is isomorphic to
       $X_i$ and that $Z$ is a Cayley-Abels graph of 
       $(G,\mathcal{H}\cup\{\langle K_1,K_2\rangle\})$.


\end{enumerate}

\end{theorem}
Let us highlight that the last statement of the above theorem on 
Cayley-Abels graphs is~\cite[Theorem 3.1]{BiMa23}, our contribution here
is to extend their methods from graphs to arbitrary spaces. Theorem~\ref{thm:CombinationFine} has the following analog for
HNN-extensions.  


\begin{theorem}\label{thm:HNNFine}
Let $(G, \mathcal{H} \cup\{K,L\})$ be a pair with $K\neq L$, $C\leq K$, and  $\varphi\colon C\to L$ a group monomorphism. Let $G\ast_\varphi$ denote the HNN-extension
$\langle G, t\mid  t c t^{-1} =\varphi(c)~\text{for all $c\in C$} \rangle$.
Let $X$ be a $G$-space that has points $x$ and $y$ whose $G$-stabilizers are $K$ and $L$ respectively, and their $G$-orbits are disjoint. 
Then there is a $G\ast_\varphi$-space $Z$ with the following properties:
\begin{enumerate}[series=HNN]
    \item $Z$ has a point $z$ such that the $G$-stabilizer  $G_z=\langle K^t,L\rangle$. 
    
    \item There is a $G$-equivariant  topological embedding $\spike(X, K, L, x, y)\hookrightarrow Z$ that maps $\hat x$ to $t^{-1}z$ and $\hat y \to z$.
     \item If the $G$-action on $X$ is  cocompact, then the action of $G\ast_\varphi$  on  $Z$ is cocompact.

     \item If $X$ is connected  then $Z$ is connected.

        \item If every   $H\in\mathcal{H} \cup\{K, L\}$ is the $G$-stabilizer of a point of $X$, then every $H\in \mathcal{H}\cup\{\langle K^t,L \rangle\}$ is the $G\ast_\varphi$-stabilizer of   a point of $Z$.
    
        \item If point $G$-stabilizers in $X$ are finite or conjugates of subgroups 
              in $\mathcal{H}\cup\{ K, L\}$, 
              then point $G$-stabilizers in $Z$ are finite or conjugates of subgroups 
              in $\mathcal{H}\cup\{\langle K^t,L \rangle\}$.

        \item There is a  $G\ast_\varphi$-tree $T$ and a $G\ast_\varphi$-map $\xi\colon  Z \to T$  with the following properties:
  \begin{enumerate}
      \item $\xi (z)$ is a vertex of $T$ and the preimage $\xi^{-1}(\xi (z))$ 
            consists of only the point $z$.
      
      \item The image $\xi (X )$ is a single point $v$ of $T$. 
      
      \item The preimage $\xi^{-1}(\mathsf{Star}(v))$ is homeomorphic 
            to $\spike(X, K, L, x, y)$.
     
      \item The $G$-orbit of $\xi(z)$ and its complement in the set of vertices of 
            $T$ make $T$ a bipartite graph.
  \end{enumerate}
\end{enumerate}
Moreover,
\begin{enumerate}[resume=HNN]
  \item 
        If $X$ is a geodesic metric space and the $G$-action on $X$ is by isometries, 
        then $Z$ is a geodesic metric space, the $G\ast_\varphi$-action is by isometries,
        and the embedding $X\hookrightarrow Z$ is isometric and convex. 

 \item \label{item:HNN-relative-p.d}
       If $X$ is a geodesic metric space and the $G$-action on $X$ 
       is properly discontinuous relative to $\mathcal{H}\cup\{ K, L\}$, 
       then $G\ast_\varphi$ action of $Z$ is
       properly discontinuous relative to $\mathcal{H}\cup\{\langle K^t,L \rangle\}$.       

    \item If $X$ is a Cayley-Abels graph for $(G, \mathcal{H} \cup\{K,L\})$, then one can assume that $\xi^{-1}(\mathsf{Star}(v))$ is isomorphic to $X$ and that $Z$ is Cayley-Abels graph of $(G, \mathcal{H}\cup\{\langle K_1,K_2\rangle\})$.
\end{enumerate}
\end{theorem}
 
The last statement of the above theorem is~\cite[Theorem 3.2]{BiMa23}.
This section describes arguments proving the above
Theorems~\ref{thm:CombinationFine} and~\ref{thm:HNNFine}. As previously
stated, we extend methods for graphs in~\cite{BiMa23} to arbitrary
spaces. The extension requires the use of coned-off spaces in order to
adjust some of the arguments.

\subsection{Pushouts in the category of $G$-sets and action extensions.} Let $\phi\colon R\to S$ and $\psi\colon R \to T$ be $G$-maps. The pushout of $\phi$ and $\psi$ is defined as follows. Let $Z$ be the $G$-set  obtained as the quotient of 
the disjoint union of $G$-sets  $S\sqcup T$ by the equivalence  relation generated by all pairs   $s\sim t$ with $s\in S$ and $t\in T$ satisfying that  there is $r\in R$ such that $\phi(r)=s$ and $\psi(r)=t$.
There are canonical $G$-maps $\imath\colon S \to Z$ and $\jmath\colon T \to Z$ such that $\imath\circ \phi= \jmath \circ \psi$. This construction satisfies the universal property of pushouts in the category of $G$-sets.

\begin{proposition}\label{lem:pushout-set}\cite[Proposition 4.1]{BiMa23}
 Let $\phi\colon R\to S$ and $\psi\colon R \to T$ be $G$-maps. Consider the pushout 
  \[ 
  \begin{tikzcd} 
                    &  S \arrow{rd}{\imath} &  \\
R \arrow{ru}{\phi} \arrow{rd}{\psi} &     & Z \\                    &  T\arrow{ru}{\jmath} &
\end{tikzcd} \]
of $\phi$ and $\psi$. 
Suppose there is $r\in R$ such that $R=G.r$. If $s=\phi(r)$, $t=\psi(r)$ and $z=\imath(s)$ then the $G$-stabilizer $G_z$ equals the subgroup $\langle G_s,G_t \rangle$. 
\end{proposition}
 
In the case that $K$ is a subgroup of $G$ and $S$ is a $K$-set, there is a natural extension of the $K$-set to a $G$-set $G\times_K S$. Specifically, up to isomorphism of $K$-sets, the $K$-set $S$ is a disjoint union of $K$-sets\[ S  =  \bigsqcup_{i\in I} K/K_i \]
where  $K/K_i$ is the $K$-set consisting of left cosets of a subgroup $K_i$ of $K$.
The $G$-set $G\times_K S$ is defined as a disjoint union of $G$-sets
\[  G\times_K S := \bigsqcup_{i\in I} G/K_i .\]
The canonical $K$-map \[ \imath\colon S \to G\times_K S, \qquad K_i \mapsto K_i\]
is injective. 

 \begin{proposition}\label{lem:fromHsetstoGset} \cite[Prop. 4.2]{BiMa23}
Let $K\leq G$ and $S$ a $K$-set. 
\begin{enumerate}
\item The canonical $K$-map $ \imath\colon S \to G\times_K S$ induces a bijection of orbit spaces $S/K \to   (G\times_K S) /G$.\label{2.1-orbits}

\item For each $s\in S$, the $K$-stabilizer $K_s$ equals the $G$-stabilizer $G_{\imath(s)}$.\label{2.1-stabilizers}

\item If $T$ is a $G$-set and $f\colon S \to T$ is  $K$-equivariant, then there is a unique $G$-map $\tilde f\colon  G\times_K S \to T$ such that $\tilde f \circ \imath = f$.\label{2.1-universal} 
 
\item  If $\imath(S)\cap g.\imath(S)\neq \emptyset$ for $g\in G$, then $g\in K$ and $\imath(S)=g.\imath(S)$.
\label{2.1-disjunion}

\item  In part three, if $f$ induces an injective map $S/K \to T/G$ and  $K_s = G_{f(s)}$ for every $s\in S$, then $\tilde f$ is injective.
\end{enumerate}
\end{proposition}

We are interested in the case that $S$ is $K$-space, that is, $S$ is a topological space with a $K$-action by homeomorphisms.

 \begin{proposition}\label{lem:fromKsetToGsetTopology}
Let $K\leq G$ and $S$ a $K$-space.
Then $G\times_K S$ admits a unique topology such that the canonical $K$-map $ \imath\colon S \to G\times_K S$ is  a topological embedding, $\imath(S)$ is open in $G\times_K S$, and the $G$-action 
on $G\times_K S$ is by homeomorphisms.  
\end{proposition}
\begin{proof}
  Consider the final topology $\tau_{fin}$  on  $G\times_K S$ induced by the functions $\{g\circ \imath \mid g\in G\}$ where 
  $g\circ \imath\colon S \to G\times_K S$ is given by $g\circ \imath (s) = g.\imath(s)$. 

  Let us verify that $\imath$ is a topological embedding and $\imath(S)$ is open. By definition $\imath$ is a continuous injection.
  Let $U$ be an open subset of $S$ and let us show that $\imath(U)$ is open in $G\times_K S$. For any $g\in G$, the preimage $(g\circ \imath)^{-1}(\imath(U))$ is empty if $g\not\in K$, and equal  to $g^{-1}.U$ if $g\in K$. Since the $K$-action on $S$ is by homeomorphisms, in any case, $(g\circ \imath)^{-1}(\imath(U))$ is an open subset of $S$. Therefore, $\imath(U)$ is an open subset of $G\times_K S$.

  Now we verify that the $G$-action on $G\times_K S$ is by homemorphisms. It is enough to show that each $g\in G$ acts as an open map $G\times_KS \to G\times_KS$.
  Let $g\in G$   and $V$ and open subset of $G\times_KS$. To show that $g.V$ is open, let $h\in G$. The preimage $(h\circ \imath )^{-1}(g.V)$ equals 
 $\imath^{-1}(h^{-1}g.V)=(g^{-1}h\circ \imath)^{-1}(V)$  
which is open in $S$ since $V$ is open in $S$.

  Let us conclude showing the uniqueness of the topology.  
  Consider a topology $\tau$ on $G\times_K S$  satisfying that   $\imath$ is a topological embedding and $G$-acts by homemorphisms on $G\times_K S$. Observe that $\tau$ is coarser than the final topology $\tau_{fin}$ induced by $\{g\circ \imath \mid g\in G\}$. Conversely, let $V \subset G\times_K S$ open in the final topology.   For any $g\in G$, $g^{-1}.V \cap \imath(U))$ is open in $\tau_{fin}$, hence $\imath^{-1}(g^{-1}.V \cap \imath(U))$ is open in $S$ and since $\imath$ is a topological embedding with respect to $\tau$, we have that $ g^{-1}.V \cap \imath(U) $ is open with respect to $\tau$.   Since
  \[ V= \bigcup_{g\in G} V\cap g.\imath(U)  =\bigcup_{g\in G}g.( g^{-1}.V \cap \imath(U)), 
  \]
  we have that $V$ is open with respect to $\tau$. Therefore, $\tau$ and $\tau_{fin}$ are the same topology.
\end{proof}  

From here on, if $S$ is a $K$-space and $K\leq G$ then $G\times_K S$ is a considered a $G$-space with the topology defined by the above proposition. Recall that a $G$-space $X$ is \emph{ cocompact} if the quotient space $X/G$ is compact. The following statement is a direct consequence of the above propositions.

\begin{corollary}\label{cor:Cocompactness}
If $S$ is a cocompact $K$-space and $K\leq G$ then $G\times_K S$ is a cocompact $G$-space.
\end{corollary}



\subsection{Pushouts of spaces}
Let $ X$ and $ Y$ be $G$-spaces, let $C\leq G$ be a subgroup and suppose $ X^C$ and $ Y^C$ are non-empty. Let $x\in  X^C$ and $y\in  Y^C$ be points. 

The\emph{  {$C$-pushout $ Z$ of $ X$ and $ Y$ with respect to the pair $(x,y)$}} is the quotient $G$-space $Z$
obtained by taking the disjoint union of $X$ and $Y$ and then  identifying the points $g.x$ and $g.y$ for every $g\in G$.
\[ 
  \begin{tikzcd} 
                    &    X \arrow[rrd, bend left, "\jmath_1"] \arrow{rd}{\imath_1} & &  \\
G/C \arrow[ru, "\kappa_1"] \arrow[rd, "\kappa_2"' ]  &     &  Z\ar[r, dashed] &  W\\                    &   Y\arrow[ ru, "\imath_2"'] \arrow[rru, bend right, "\jmath_2"'] & &
\end{tikzcd} \]
The standard universal property of pushouts holds for this construction: if $\jmath_1\colon  X \to  W$ and $\jmath_2\colon  Y \to  W$ are  $G$-maps such that  $\jmath_1\circ \kappa_1 = \jmath_2\circ \kappa_2$, then there is a unique   $G$-map $ Z\to  W$ such that above diagram commutes. 

\begin{remark}
Let $ Z$ be the $C$-pushout of $ X$ and $ Y$ with respect to   $(x_0,y_0)$.  For any point $x$ in $ X$,   $G_x = G_{\imath_1(x)}$ or $x$ is in the image of $\kappa_1$.  
\end{remark}

The analogous of the following proposition in the case that $X$ and $Y$ are graphs 
is~\cite[Prop. 4.5]{BiMa23}. The proof of the cited result hold mutatis mutandis for all items of the proposition below. Note that the second and fifth item below follow  directly from the definition of $C$-pushout as a quotient space, and our   Proposition~\ref{lem:fromKsetToGsetTopology} and Corollary~\ref{cor:Cocompactness} respectively. 

\begin{proposition}\label{lem:pushout}~\cite[Prop. 4.5]{BiMa23}
Let $G$ be the amalgamated free product group $A\ast_C B$, let  ${X}$ and $Y$ be an $A$-space and  a $B$-space respectively. Let $x\in {X}^C$    and $y\in {Y}^C$ and let ${Z}$ be the $C$-pushout of $G\times_{A} {X}$ and $G\times_{B} {Y}$  with respect to $( x, y)$. Let $z=\imath_1(x)=\imath_2(y)$. The following properties hold:
\begin{enumerate}
 \item The homomorphism $A_x \ast_C B_y \to G$ induced by the inclusions $A_x\leq G$ and $B_y\leq G$ is injective and has an image $G_z$. In particular $G_z=\langle A_x, B_y \rangle$ is isomorphic to $A_x \ast_C B_y$.\label{married}

    \item The $A$-map ${X}\hookrightarrow G\times_A {X} \xrightarrow{\imath_1} {Z}$
    is a topological  embedding. The analogous statement for $ {Y}\hookrightarrow G\times_B {Y} \xrightarrow{\imath_2} {Z}$
    holds. 
    
 \item[ ] From here on, we consider $ X$ and $ Y$ as subspaces of $ Z$ via these canonical embeddings.
\item For any $w\in Z$, there is $g\in G$ such that $g.w$  is an element  of  $ X \cup {Y}$. \label{transfer} 
\item For every  $w\in{X}$ which is not in the $A$-orbit of $x$, $A_w = G_w$ where $G_w$ is the $G$-stabilizer of $w$ in $ Z$. The analogous statement holds for $w\in Y$.\label{singles}

\item If $X$ is a cocompact $A$-space and $Y$ is a cocompact $B$-space, then   $Z$ is a cocompact $G$-space.  \label{cocompact}
  \item If ${X}$ and ${Y}$ are connected, then ${Z}$ is connected.\label{connected}
\end{enumerate}
\end{proposition}

Under mild modifications on the spaces $X$ and $Y$, the $G$-space $Z$ from the previous proposition  admits a natural map $\xi\colon Z \to T$ into a Bass-Serre tree $T$ of an splitting of $G$, this is the content of the next proposition. The map $\xi$ will be our main tool in some of the remaining sections of the article.    

\begin{proposition}\label{prop:xi-map}
  Let $G$ be the amalgamated free product group $A\ast_C B$, let  ${X}$ and $Y$ be a (metric) $A$-space and  a (metric) $B$-space respectively. Let $x\in {X}^C$    and $y\in {Y}^C$, $\hat X = \spike(X, C, x) $, $\hat Y=\spike(Y, C, y)$ and let ${Z}$ be the $C$-pushout of $G\times_{A} \hat X$ and $G\times_{B} \hat Y$  with respect to $(\hat x, \hat y)$. Let $z=\imath_1(\hat x)=\imath_2(\hat y)$. Then there is a  $G$-tree $T$ and a $G$-map $\xi\colon  Z \to T$ of  with the following properties:
  \begin{enumerate}
      \item $\xi (z)$ is a vertex of $T$ and the preimage $\xi^{-1}(\xi (z))$ consists of only the point $z$.
      
      \item The image $\xi (X)$ is a single vertex of $T$ and analogously the image of $\xi (Y)$ is a single vertex of $T$.

      \item The restriction of $\xi$ to the open interval between $x$ and $\hat x$ is a topological embedding.

      \item The $G$-orbit of $\xi(z)$ and its complement in the set of vertices of $T$ make $T$ a bipartite graph.

      \item If a vertex $v$ of $T$ is not in the $G$-orbit of $\xi(z)$, then the preimage of the star of $v$ is a subspace of $Z$ of the form $g.\hat X$ or $g.\hat Y$ for some $g\in G$. 
  \end{enumerate}
 Moreover,  if $X$ and $Y$ are geodesic metric spaces and the corresponding $A$ and $B$ actions are by isometries, then $Z$ is a geodesic metric space, the $G$-action is by isometries, and the embeddings $X\hookrightarrow Z$ and $Y\hookrightarrow Z$ are isometric and convex.
\end{proposition}
\begin{proof} 
First, note that all statements in Proposition~\ref{lem:pushout} hold for $Z$ as a $C$-pushout of  $G\times_A  \hat X$ and $G\times_B \hat Y$ with respect to $(\hat x, \hat y)$.  

Let us define the tree $T$ and the map $\xi$.   The natural embeddings $X\hookrightarrow \hat X$ and $\hat X \hookrightarrow G\times_A \hat X$ show that $X$ is canonically identified with an $A$-equivariant subspace of $G\times_A \hat X$. Analogously, there is a canonical identification of $Y$ with a $B$-equivariant subspace of $G\times_B \hat Y$.

Consider the splitting $G=A\ast_{A_x}(A_x\ast_CB_y)\ast_{B_y}B$ and observe that the subgroups $A_x$, $B_y$ and $A_x\ast_C B_y$ are naturally identified with the $G$-stabilizers of $\hat x\in G\times_A \hat X$, $\hat y\in G\times_B  \hat Y$, and $z\in Z$. Let $T$ denote the Bass-Serre tree of this splitting. The vertex and edge sets of $T$ can be described as 
\[ V(T) = G/A \sqcup G/(A_x\ast_C B_y)  \sqcup G/B \]
and 
\[ E(T) = \{ \{ gA, g(A_x\ast_C B_y) \} \mid g\in G \} \sqcup \{ \{g(A_x\ast_C B_y), gB\} \mid g\in G \}\]
respectively. Note that $T$ is a bipartite $G$-graph, the equivariant bipartition of the vertices given by $G/A\sqcup G/B$ and $G/(A_x\ast_CB_y)$. 

Now we define a $G$-map $\xi\colon Z \to T$ using the universal property of $C$-pushouts. Consider the $A$-map from $\hat X$ to $T$ that maps the $A$-subspace $X$ to the vertex $A$, the point $\hat x$ to the vertex $A_x\ast_C B_y$ and the open interval between $x$ and $\hat x$ (isometrically) homeomorphically to the open edge $\{A, A_x\ast_C B_y\}$. This induces a $G$-map  $\jmath_1\colon G\times_A  \hat X \to T$. Analogously, there is $B$-map $ \hat Y\to T$ such maps 
 the $B$-subspace $Y$ to the vertex $B$, the point $\hat y$ to the vertex $A_x\ast_C B_y$ and the open edge between $y$ and $\hat y$ to the edge $\{B, A_x\ast_C B_y\}$; this $B$-maps  induces a unique $G$-map $\jmath_2\colon G\times_B  Y \to T$. 
\[ 
  \begin{tikzcd} 
                    &   G\times_A \hat X \arrow[rrd, bend left, "\jmath_1"] \arrow{rd}{\imath_1} & &  \\
G/C \arrow[ru, "\kappa_1"] \arrow[rd, "\kappa_2"' ]  &     & Z\ar[r, dashed, "\xi"] & T\\                    &  G\times_B \hat Y\arrow[ ru, "\imath_2"'] \arrow[rru, bend right, "\jmath_2"'] & &
\end{tikzcd} \] 
Consider the $G$-maps $\kappa_1 \colon G/C \to G\times_A  \hat X$ and $\kappa_2\colon G/C \to G\times_B  \hat Y$ given by $C\mapsto \hat x$ and $C\mapsto \hat y$ respectively. Since $\jmath_1\circ \kappa_1 = \jmath_2\circ \kappa_2$,  the universal property  implies that there is a surjective $G$-map $\xi\colon  Z\to T$. The verification of the five items in the statement of the proposition is mutatis mutandis the one for~\cite[Prop. 4.5(8)]{BiMa23} after replacing the word subgraph by subspace.  

Now we prove the last part of the proposition. Suppose that $X$ and $Y$ are geodesic metric spaces and the corresponding actions are by isometries. Since $X$ is a geodesic metric space, $\spike(X, C, x)$ is a geodesic metric space with $X$ an isometrically embedded convex subspace. The analogous remarks  for  $\spike(Y, C, y)$ and $Y$ hold as well. 

The metric on $Z$ is defined locally, and then the global metric is
defined as the induced path metric. Locally, every point $w$ of $Z$ that
is not in the $G$-orbit of $z$ is contained in a subspace of the form
$g.\hat X$ or $g.\hat Y$. On the other hand, there is a neighborhood $W$
of $z$ in $Z$ isometric to the geodesic metric space of diameter two
obtained as the quotient of $((A_x\ast_C A_y)/C)\times (0,1]$ by the
relation $\sim$ which identifies all elements with second coordinate
equal to one. Note that the intersection $\hat X \cap W$ is the half
open segment of length one between $x$ and $z$, closed at $z$ on which
the metrics on $\hat X$ and $W$ coincide. Analogously, the metrics on
$\hat Y$ and $W$ coincide in the intersection $\hat Y \cap W$. The map
$\xi\colon Z \to T$ shows that if a path in $Z$ has an initial and
terminal point in $\hat X$, then either the path is embedded and stays
inside $\hat X$; or the path intersects the complement of $\hat X$ and
is not embedded. The analogous statement holds for $\hat Y$. It follows
that if $\gamma$ is a shortest path between two points of $Z$, then
$\gamma \cap g.\hat X$ is either empty or a geodesic in $g.\hat X$, and
analogously for $g.\hat Y$. Since $T$ is a tree, it follows that between
any two points of $Z$ there is a unique shortest path between them, and
therefore $Z$ is a geodesic space and moreover $g.\hat X$ and $g.\hat Y$
are convex subspaces of $Z$ for every $g\in G$.
\end{proof}

We conclude the section with the proof of the main result of the section. 

\begin{proof}[Proof of Theorem~\ref{thm:CombinationFine}]
 Let $X$ be the $C$-pushout of the $G$-spaces $\hat X_1= G\times_{G_1}
\spike(X_1, C, x_1)$ and $\hat X_2= G\times_{G_2}
\spike(X_1, C, x_2)$ with respect to $(x_1,x_2)$, and let $z$ be
the image of $x_1$ in $\Gamma$. 
Then the statements except for 
(\ref{item:Amalgamation-relative-p.d}) 
follow directly from
Propositions~\ref{lem:pushout} and~\ref{prop:xi-map}.
\end{proof}

\input{35-Amalgamated-relative-p-d-actions}

\subsection{Coalescence of spaces for HNN-extensions} 

\begin{definition}
\label{defn-coalescence}(Coalescence in sets)~\cite{BiMa23}
 Let $H$ be a subgroup of a group $A$,  
 let $\varphi\colon H\to A$ be a monomorphism and let $G$ be the HNN-extension
 \[ 
  G=A\ast_{\varphi}
   =\langle G, t\mid  t c t^{-1} =\varphi(c)~\text{for all $c\in H$} \rangle .
\]  
 Let $ X$ be an $A$-space, $x\in X^H$ and $y\in  X^{\varphi(H)}$. 
 The {\emph{ $\varphi$-Coalescence of $X$ with respect to $(x,y)$}} 
 is the $G$-space $Z$ arising as quotient  of $G\times_A  X$ 
 by the equivalence relation generated by
 \[ 
  \mathcal{B}=\{ (gt.x, g.y) \mid  \text{ $g\in G$} \}. 
 \] 
 Note that the quotient map 
 \[ \rho\colon G\times_A X \to Z \] 
 is $G$-equivariant.  
 We use the map $\rho$ in the statements of the following propositions.
\end{definition}

The following is the analogous of Proposition~\ref{lem:pushout} for HNN-extensions.

\begin{proposition}
\cite[Proposition 5.5]{BiMa23}\label{lem:HNN}
Let $H$ be a subgroup of a group $A$,  let $\varphi\colon H\to A$ be a monomorphism and let $G=A\ast_{\varphi}$.  Let $ X$ be an $A$-space, let $x,y\in  X$ in different $A$-orbits such that $A_x=H$,   $A_y=\varphi(H)$.   If
$ Z$ is the $\varphi$-Coalescence of $ X$ with respect to $(x,y)$, and $z=\rho(y)$. The following properties hold:

\begin{enumerate}
 \item The $G$-stabilizer $G_z$ of $z$ equals the subgroup $\varphi(H)$.\label{identified vertices-HNNN}

    \item The $A$-map  ${X}\hookrightarrow G\times_A {X} \xrightarrow{} {Z}$
    is a topological  embedding.
    
 \item[] From here on, we consider $ X$  as a subspace of $ Z$ via this canonical embedding.

 \item For every $w \in Z$, there is $g\in G$ such that $g.w$  is an element of $ X$. \label{transfer-HNN-}
 
\item For $w\in {X}$ which is not in the $A$-orbit of $x$, $A_w = G_w$ where $G_w$ is the $G$-stabilizer of $w$ in $ Z$.

\item If $X$ is a cocompact $A$-space, then $Z$ is a cocompact $G$-space.

\item If $X$ is connected, then $Z$ is connected.
 \end{enumerate}
\end{proposition}

There is an analogous statement to Proposition~\ref{prop:xi-map} for a coalescence.

\begin{proposition}
Let $H$ be a subgroup of a group $A$,  let $\varphi\colon H\to A$ be a monomorphism and let $G=A\ast_{\varphi}$.
Let $ X$ be an $A$-space, let $x,y\in  X$ in different $A$-orbits such that $A_x=H$,   $A_y=\varphi(H)$. Let $Z$ be the $\varphi$-coalescence of $\hat X = \spike(X,H,\varphi(H),x,y)$ with respect to $(\hat x, \hat y)$, and let $z=\rho(\hat y)$.

Then there is a  $G$-tree $T$ and a $G$-map $\xi\colon  Z \to T$ of  with the following properties:
  \begin{enumerate}
      \item $\xi (z)$ is a vertex of $T$ and the preimage $\xi^{-1}(\xi (z))$ consists of only the point $z$.
      
      \item The image $\xi (X)$ is a single vertex of $T$. 

      \item The restriction of $\xi$ to the open interval between $x$ and $\hat x$ is a topological embedding. The analogous statement for the open interval between $y$ and $\hat y$ holds. 

      \item The $G$-orbit of $\xi(z)$ and its complement in the set of vertices of $T$ make $T$ a bipartite graph.

      \item If a vertex $v$ of $T$ is not in the $G$-orbit of $\xi(z)$, then the preimage of the star of $v$ is a subspace of $Z$ of the form $g.\hat X$  for some $g\in G$. 
  \end{enumerate}
  Moreover, if $X$ is a geodesic space and the $A$-action is by isometries, then $Z$ is a geodesic metric space, the $G$-action is by isometries, and the embedding  $X\hookrightarrow Z$ is isometric and convex.
\end{proposition}
\begin{proof}
Below we describe the definition of the tree $T$ and the $G$-map $\xi$. 
The argument proving the items of the proposition is essentially~\cite[Proof of Prop. 5.5(8)]{BiMa23}, and the proof of the moreover part is analogous to the  corresponding argument in the proof of Proposition~\ref{prop:xi-map}. 

Let $T$ be the barycentric subdivision of the Bass-Serre tree of $G\ast_{\varphi}$. Observe that  
\[ V(T) = G/A \sqcup  G/H \]
and 
\[ E(T) = \{ \{ gA, gtH \} \mid g\in G \} \sqcup \{ \{gA, gH\} \mid g\in G \}.\] 
Then there is an induced  $G$-equivariant map \[\psi\colon G\times_A   \spike(X,H,\varphi(H),x,y)   \to T\] 
such that 
all $X$ is mapped to the vertex $A$, $\hat x \mapsto H$, $\hat y \mapsto tH$,   the interval between $x$ and $\hat x$ is mapped to the open edge $\{A, H\}$, and the interval between $y$ and $\hat y$ is mapped to the edge $\{A, tH\}$. Now the map $\psi$ induces a  $G$-equivariant morphism of graphs  $\xi \colon   Z\to T$ such that   
\[ \begin{tikzcd} 
 & G\times_A   X \arrow["\psi"]{dr} \arrow[swap,"\rho"]{dl}&  \\
  Z \ar[dashed, "\xi"]{rr}   & & T. 
\end{tikzcd}
\]
is a commutative diagram.
\end{proof}

\begin{proof}[Proof of Theorem~\ref{thm:HNNFine}]
The argument is the same as the one used to prove Theorem~\ref{thm:CombinationFine}, the only difference is the use of  Proposition~\ref{lem:HNN} instead of Proposition~\ref{lem:pushout}.
\end{proof}

%% file: 35-Amalgamated-relative-p-d-actions.tex
 \begin{proof}[Proof of (\ref{item:Amalgamation-relative-p.d}) in \cref{thm:CombinationFine}]
 We will apply Lemma \ref{lem:criteria-p.d.action}.
 Let $z\in Z$ be the base point given in (\ref{item:Amalgamation-basepoint})
 in \cref{thm:CombinationFine}. 
 We fix $r>0$. Let $g\in G$ such that
 $d_Z(z,g.z)\leq r$. Suppose that 
 \begin{align*}
     g=a_1b_1a_{2}b_2\cdots a_kb_kc
 \end{align*}
where $a_i\in G_1\setminus C$ and $b_i\in G_2\setminus C$, 
for $1\leq i\leq k$, and $c\in C$. 
Since $c.z=z$, we have $g.z=a_1b_1a_{2}b_2\cdots a_kb_k. z$.
Let $\ell$ be the constant given by  \cref{thm:CombinationFine} \eqref{item:Amalgamation-geometric-embedding-CF} and observe that $k\ell\leq r$.

 Let $\gamma\colon [0,1]\to Z$ be a geodesic from $z$ to $g.z$.
 By the construction of geodesics in $Z$ and the construction of the 
 action of $G$ on $Z$, we have the following.
 \begin{claim}
  \label{claim:amalgamation-path}
   There exists a sequence
  $0< t_1\leq s_1\leq t_2\leq s_2\leq \dots t_k\leq s_k=1$ such that,
  for all $i$,
 \begin{itemize}
  \item $\gamma(t_i) = a_1b_1\dots a_i.z$, 
  \item $\gamma(s_i) = a_1b_1\dots a_ib_i.z$,
  \item $\gamma([0,t_{1}])\subset \spike(X_1, C, x_1,l)$,
  \item $\gamma([s_{i-1},t_{i}])\subset a_1b_1\dots b_{i-1}.\spike(X_1, C, x_1,l)$,
  \item $\gamma([t_i,s_i])\subset a_1b_1\dots b_{i-1}a_i.\spike(X_2, C, x_2,l)$. 
 \end{itemize}
 \end{claim}

 Let $S_i\subset G_i$ be a finite relative generating set for 
 $(G,\mathcal{H}_i\cup\{K_i\})$. We denote by $\dist_i$ the graph distance of
 the coned-off Cayley graph $\hat\Gamma(G_i,\mathcal{H}_i\cup\{K_i\},S_i)$.

 Let $\Lambda_i=\Lambda_i[\hat{x}_i,S_i]\colon \R_{\geq 0}\to \N$ 
 be the non-decreasing function given in
 Lemma~\ref{lem:rel-p.d.-coarse-embedding} for the action of $G_i$ 
 on $\spike(X_i, C, x_i)$ with the base point $\hat{x}_i$. 
 Here, we can identify
 $\hat{x}_i$ with $z$ 
 by (\ref{item:Amalgamation-topological-embedding-CF}) in 
 \cref{thm:CombinationFine}. 

 We denote by $d_2$ the metric on $\spike(X_2, C, x_2,l)$. 
 Since 
 \begin{align*}
  d_{2}(z,b_i.z) =  d_Z(a_1b_1\dots a_i.z,a_1b_1\dots a_ib_i.z) =
  d_Z(\gamma(t_i),\gamma(s_i)) \leq r,
 \end{align*}
 we have $\dist_2(1,b_i)\leq \Lambda_2(r)$. Similarly, we have
 $\dist_1(1,a_i)\leq \Lambda_1(r)$.

 Let $S=S_1\cup S_2$. We denote by $\dist_{\hat\Gamma}$ the 
 graph distance of
 the coned-off Cayley graph 
 $\hat\Gamma(G,\mathcal{H}\cup \{\langle K_1,K_2 \rangle\},S)$. 
 It follows that
 \begin{align*}
  \dist_{\hat\Gamma}(1,g)\leq \sum_{i=1}^k
  \left(
    \dist_{\hat\Gamma}(1,a_{i}) 
     + \dist_{\hat\Gamma}(1,b_i)  \right) + \dist_{\hat\Gamma}(1,c)
  \leq \frac{r}{l}(\Lambda_1(r)+\Lambda_2(r)) + 1
 \end{align*}
 Therefore, the set $\{g\in G \mid U\cap g.U \neq \emptyset\}$ has a 
 finite diameter relative to $\mathcal{H}\cup \{\langle K_1,K_2 \rangle\}$.
 By applying Lemma~\ref{lem:criteria-p.d.action}, it follows 
 that the action of $G$ on $Z$ is properly discontinuous relative to 
 $\mathcal{H}\cup \{\langle K_1,K_2 \rangle\}$
\end{proof}

%% file: 40-BicombingCombination.tex

The main result of this section constructs a bicombing on a geodesic metric space $(Z,d_{\mathsf{Z}})$ from a tree of spaces $\xi\colon Z\to T$ and a bicombing on each vertex space $X_k$.  The resulting bicombing can be assumed to be equivariant under natural assumptions defined below.  
 
Let $\xi\colon Z \to T$ be a $G$-tree of spaces and suppose that for each $k\in K$ there is a bicombing $\Gamma_k$ on $X_k$. We say that \emph{the family $\{\Gamma_k\}_{k\in K}$ is $G$-equivariant} if 
for any $h \in G$, $k \in K$, and $x,y \in X_k$,
    \begin{align*}
    h \cdot {\Gamma}_{k}(x,y)(t)= {\Gamma}_{hk}(hx,hy)(t)
    \end{align*}
 holds for $t \in [0,1]$.

\begin{theorem}\label{TS-Main_thm}
Let $Z$ be a tree of spaces associated with a map $\xi\colon Z\to T$, where $V(T)=K \sqcup L$. 
Suppose that each $X_k=\xi^{-1}(\starNb(k))$ admits a bicombing ${\Gamma}_k \colon X_k \times X_k \times [0,1] \to X_k$.
There is a bicombing $\Gamma \colon Z  \times Z \times [0,1] \to Z$ on $Z$ satisfying the following properties.

\begin{enumerate}
\item \label{TS-Main1}
The  restriction of $\Gamma$ to 
$X_k \times X_k \times [0,1]$ coincides with ${\Gamma}_k$.

\item \label{TS-main-qg}
If ${\Gamma}_k$ is $(\lambda,c)$-quasi-geodesic bicombing for every $k\in K$ 
then $\Gamma$ is $(\lambda,2c)$-quasi-geodesic bicombing.
In particular, if ${\Gamma}_k$ is a geodesic bicombing for every $k\in K$,
then $\Gamma$ is a geodesic bicombing.

\item \label{TS-main-consist}
If ${\Gamma}_k$ is consistent geodesic bicombing for every $k\in K$ 
then $\Gamma$ is consistent geodesic bicombing.

\item \label{TS-main-bdd-bcomb}
Let $\lambda, c_1\geq1$ and $k, c_2\geq0$. If each ${\Gamma}_k$ is a $(\lambda,k,c_1,c_2)$-bounded bicombing, 
then $\Gamma$ is a $(\lambda,2k,\lambda c_1,2k+c_1+c_2)$-bounded bicombing. 

\item \label{TS-main-gcc}
Let $E\geq1$ and $C>0$. If each ${\Gamma}_k$ is a $(E,C)$-geodesic coarsely convex bicombing, then $\Gamma$ is a $(E+2,6C)$-geodesic coarsely convex bicombing.
\end{enumerate}

Moreover, suppose $\xi: Z \to T$ is a $G$-tree of spaces.
\begin{enumerate}
\setcounter{enumi}{5}
\item \label{TS-main-G-equiv}
If $\{{\Gamma}_k\}_{k\in K}$ is $G$-equivariant
then the bicombing $\Gamma$ is $G$-equivariant
. 

\end{enumerate}
\end{theorem}

Let us remark that 
item~\eqref{TS-main-bdd-bcomb} in the above theorem was proved by Alonso and Bridson~\cite{AB95}. From here on we assume the hypothesis of Theorem~\ref{TS-Main_thm}, and we prove the result in the following subsections.

\subsection{Construction of a bicombing on a tree of spaces}
In this subsection, we define the bicombing of the space $Z$ from the bicombings on the vertex spaces. Specifically, a map $\Gamma\colon Z\times Z \times [0,1] \to Z$.  The proofs that the bicombing on $Z$ inherits properties of the bicomings on the vertex spaces are the content of the following subsections.   
 
 Let $v,w \in Z$. 
Let $[\xi(v), \xi(w)]$ be the geodesic segment on $T$ connecting $\xi(v)$ and $\xi(w)$, and let $\xi(v), x_1, \dots, x_{n-1}, \xi(w)$ be a sequence of vertices in this geodesic path.  
Since $T$ is a bipartite tree,   if $x_i \in K$ then $x_{i+1} \in L$, and analogously,  if $x_i \in L$, then $x_{i+1} \in K$.   
Let $\ell_1, \ell_2, \dots, \ell_m$ be the sequence of vertices of $T$ in $L$ defined by removing the vertices in $K$ of the sequence, let $v_i= \xi^{-1}(\ell_i)$, and let $k_i\in K$ be the vertex such that $v_i, v_{i+1} \in X_{k_i}$.
Let $v_0=v$ and $v_{m+1}=w$. 
Define $\Gamma(v,w)$ to be the path obtained by concatenating the paths ${\Gamma}_{k_i}(v_i,v_{i+1})$ at $\{v_i\}_{i=1}^{m-1}$, 
that is,
\begin{align}
\label{eq:GammaUV}
\Gamma(v,w)(t)\coloneqq 
{\Gamma}_{k_{i}}(v_i,v_{i+1})\left( \left( t-\dfrac{t_i}{t_{m+1}}\right ) \cdot \dfrac{t_{m+1}}{t_{i+1}-t_i} \right) 
\ 
\text{for}
\ \dfrac{t_i}{t_{m+1}} \leq t \leq \dfrac{t_{i+1}}{t_{m+1}}.
\end{align}
where $t_0=0$ and $t_i=\displaystyle\sum_{l=0}^{i-1}d_{k_l}(v_l,v_{l+1})$.

\subsection{Combination of quasi-geodesic bicombings}
\label{TS-bicombing_on_Z}
In this subsection, we prove part \eqref{TS-main-qg} of Theorem \ref{TS-Main_thm}.
For $v,w \in Z$, consider the path $\Gamma(v,w)$ as defined in the previous section. Let $0\leq t<s\leq 1$.  We will show that 
\begin{align}
\label{lcqgeod}
d_{\mathsf{Z}}(\Gamma(v,w)(t), \Gamma(v,w)(s))
\leq \lambda (s-t)d_{\mathsf{Z}}(v,w)+ 2c
\end{align}
Suppose  that $\dfrac{t_i}{t_{m+1}} \leq t \leq \dfrac{t_{i+1}}{t_{m+1}}$ and $\dfrac{t_j}{t_{m+1}} \leq s \leq \dfrac{t_{j+1}}{t_{m+1}}$. 
If $i=j$, then~\eqref{lcqgeod} holds since ${\Gamma}_{k_i}$
is $(\lambda,c)$-quasi-geodesic.
Now let us consider the case  $j>i+1$, and observe
\begin{align*}
&\hspace{-5mm}d_{\mathsf{Z}}(\Gamma(v,w)(t), \Gamma(v,w)(s)) \\
=\ &
d_{\mathsf{Z}}
\left(
{\Gamma}_{k_{i}}(v_i,v_{i+1})\left( \left( t-\dfrac{t_i}{t_{m+1}}\right ) \cdot \dfrac{t_{m+1}}{t_{i+1}-t_i} \right),
{\Gamma}_{k_{j}}(v_j,v_{j+1})\left( \left( s-\dfrac{t_j}{t_{m+1}}\right ) \cdot \dfrac{t_{m+1}}{t_{j+1}-t_j} \right) 
\right)   \\
=\ &
d_{\mathsf{Z}}
\left(
{\Gamma}_{k_{i}}(v_i,v_{i+1})\left( \left( t-\dfrac{t_i}{t_{m+1}}\right ) \cdot \dfrac{t_{m+1}}{t_{i+1}-t_i} \right),
{\Gamma}_{k_{i}}(v_i,v_{i+1})(1)
\right)\\
&+
\sum_{l=i}^{j-2}
d_{\mathsf{Z}}
\left(
{\Gamma}_{k_{l+1}}(v_{l+1},v_{l+2})(0),
{\Gamma}_{k_{l+1}}(v_{l+1},v_{l+2})(1)
\right) \\
&+
d_{\mathsf{Z}}
\left(
{\Gamma}_{k_{j}}(v_j,v_{j+1})(0),
{\Gamma}_{k_{j}}(v_j,v_{j+1})\left( \left( s-\dfrac{t_j}{t_{m+1}}\right ) \cdot \dfrac{t_{m+1}}{t_{j+1}-t_j} \right)
\right) \\
\leq\ &
\lambda \left| 1-\left( t-\dfrac{t_i}{t_{m+1}}\right ) \cdot \dfrac{t_{m+1}}{t_{i+1}-t_i} \right| d_{\mathsf{Z}}(v_i,v_{i+1}) +c
+\sum_{l=i}^{j-2} d_{\mathsf{Z}}(v_{l+1},v_{l+2}) \\
&+ 
\lambda \left| \left( s-\dfrac{t_j}{t_{m+1}}\right ) \cdot \dfrac{t_{m+1}}{t_{j+1}-t_j} \right| d_{\mathsf{Z}}(v_j,v_{j+1}) +c \\ 
=\ &
\lambda \left( \dfrac{t_{i+1}-tt_{m+1}}{t_{i+1}-t_i} \right) (t_{i+1}-t_i) + c 
+\sum_{l=i}^{j-2} (t_{l+2}-t_{l+1}) 
+\lambda \left( \dfrac{st_{m+1}-t_j}{t_{j+1}-t_j} \right) (t_{j+1}-t_j) + c \\
\leq\ &
\lambda (t_{i+1}-tt_{m+1})+ c 
+\sum_{l=i}^{j-2} \lambda(t_{l+2}-t_{l+1}) 
+ \lambda (st_{m+1}-t_j) + c \\
=\ &
\lambda (s-t)t_{m+1}+ 2c \\
=\ &
\lambda (s-t)d_{\mathsf{Z}}(v,w)+ 2c
\end{align*}
holds. The case that $j=i+1$ follows from the above inequalities by interpreting the sums of the form $\displaystyle\sum_{l=i}^{j-2}$ as zeros.  By a similar argument, we also have
\begin{align*}
\lambda^{-1} (s-t)d_{\mathsf{Z}}(v,w)- 2c \leq d_{\mathsf{Z}}(\Gamma(v,w)(t),\Gamma(v,w)(s)) .
\end{align*}
Therefore, $\Gamma(v,w)$ is a $(\lambda,2c)$-quasi-geodesic bicombing.

\subsection{Combination of consistent geodesic bicombings}
In this section we prove part \eqref{TS-main-consist} of Theorem~\ref{TS-Main_thm}. Assume $\Gamma_k$ is a consistent geodesic bicombing for every $k$ and, by 
Theorem \ref{TS-Main_thm}(\ref{TS-main-qg}), that $\Gamma$ is a geodesic bicombing.

Let $v,w \in Z$
and let $w'=\Gamma(v,w)(t)$ for some $t \in [0,1]$.
The geodesic $[\xi(v), \xi(w')]$ is an initial segment of the geodesic $[\xi(v), \xi(w)]$ in the tree $T$. Therefore, the sequence of vertices $\ell_1,\ldots , \ell_m$ in $L$ of $[\xi(v), \xi(w')]$ is an initial subsequence of the sequence of vertices $\ell_1,\ldots ,\ell_m,\ldots \ell_n$ in $L$ of 
$[\xi(v), \xi(w')]$. As in the definition of $\Gamma$, let $v_i=\xi^{-1}(\ell_i)$ and $v_0=v$. Then $\Gamma(v,w')$ is the concatenation of the paths $\Gamma(v_0,v_1), \ldots, \Gamma(v_{m-1}, v_m)$ and $\Gamma(v_m,w')$, and analogously, $\Gamma(v,w)$ is the concatenation of $\Gamma(v_0,v_1), \ldots, \Gamma(v_{n-1}, v_n)$ and $\Gamma(v_n,w)$.  Let $k$ such that $w'\in X_k$. Let us assume that $m<n$, the case $m=n$ is analogous.  Since $\Gamma_k$ is consistent, we have that $\Gamma(v_m,w')=\Gamma_k(v_m,w')$ is an initial segment of $\Gamma(v_m, v_{m+1})=\Gamma_k(v_m,v_{m+1})$. Therefore, $\Gamma(v,w')$ is an initial subsegment of $\Gamma(v,w)$.

\subsection{Combination of geodesic coarsely convex bicombing}
In this subsection we prove part \eqref{TS-main-gcc} of Theorem~\ref{TS-Main_thm}. We show that $\Gamma$ is geodesic coarsely convex by invoking Theorem~\ref{consistent-and-thin-cconvex} which has three assumptions that are verified in the following three lemmas.

\begin{lemma}
\label{TS-final}
We suppose that each ${\Gamma}_k$ is a geodesic $(E,C)$-coarsely convex bicombing.
Then $\Gamma$ is $C$-coarsely consistent.  
\end{lemma}

\begin{lemma}
\label{tscft}
If each ${\Gamma}_k$ is a geodesic $(E,C)$-coarsely convex bicombing,
then $\Gamma$ is coarsely forward thin.
\end{lemma}

\begin{lemma}
\label{tscbt}
If each ${\Gamma}_k$ is a geodesic $(E,C)$-coarsely convex bicombing,
then $\Gamma$ is coarsely backward thin.
\end{lemma}

\begin{proof}[Proof of Lemma~\ref{TS-final}]
Let $v,w \in Z$
and let $w'=\Gamma(v,w)(t)$ for some $t \in [0,1]$.
The geodesic $[\xi(v), \xi(w')]$ is an initial segment of the geodesic $[\xi(v), \xi(w)]$ in the tree $T$. Therefore, the sequence of vertices $\ell_1,\ldots , \ell_m$ in $L$ of $[\xi(v), \xi(w')]$ is an initial subsequence of the sequence of vertices $\ell_1,\ldots ,\ell_m,\ldots \ell_n$ in $L$ of 
$[\xi(v), \xi(w')]$. As in the definition of $\Gamma$, let $v_i=\xi^{-1}(\ell_i)$ and $v_0=v$. Then $\Gamma(v,w')$ is the concatenation of the paths $\Gamma(v_0,v_1), \ldots, \Gamma(v_{m-1}, v_m)$ and $\Gamma(v_m,w')$, and analogously, $\Gamma(v,w)$ is the concatenation of $\Gamma(v_0,v_1), \ldots, \Gamma(v_{n-1}, v_n)$ and $\Gamma(v_n,w)$.  Let $k$ such that $w'\in X_k$. Let us assume that $m<n$, the case $m=n$ is analogous.  

Let $c_1,c_2\in [0,1]$ such that $\Gamma_k(v_m,w')(c_1)=\Gamma(v,w')(c)$ and
$\Gamma_k(v_m,v_{m+1})(c_2)=w'$. We have
\begin{align*}
    d_{\mathsf{Z}}(v,\Gamma(v,w)(ct)) 
    &= d_{\mathsf{Z}}(v,v_m) + c_1d_{\mathsf{Z}}(v_m,w')\\
    &= d_{\mathsf{Z}}(v,v_m) + c_1 c_2 d_{\mathsf{Z}}(v_m,v_{m+1})
\end{align*}
So we have 
$\Gamma(v,w)(ct)={\Gamma}_{k}(v_m,v_{m+1})(c_1c_2)$.

Since ${\Gamma}_k$ is a $(E,C)$-geodesic coarsely convex bicombing, 
\begin{align}
d_{\mathsf{Z}}(\Gamma(v,w)(ct), \Gamma(v,w')(c)) 
&= d_{\mathsf{Z}}({\Gamma}_{k}(v_m,v_{m+1})(c_1c_2),{\Gamma}_{k}(v_m,w^{\prime})(c_1)) \notag \\
&\leq c_1Ed_{\mathsf{Z}}({\Gamma}_{k}(v_m,v_{m+1})(c_2),{\Gamma}_{k}(v_m,w^{\prime})(1))+C \notag \\
&=c_1Ed_{\mathsf{Z}}(w^{\prime},w^{\prime}) + C=C.  \notag
\end{align}
This completes the proof. 
\end{proof}

\begin{proof}[Proof of Lemma~\ref{tscft}]
For $v,w_1,w_2 \in Z$, 
let $\gamma_1=\Gamma(v,w_1)$ and $\gamma_2=\Gamma(v,w_2)$.
Suppose 
\begin{align*}
    \min\{d(v,w_1), d(v,w_2)\}\geq 1
\end{align*}
and $v,w_1,w_2$ are not on the same geodesic segment.
Let denote by $[\gamma_1(0),\gamma_1(1)]$, the image of $\gamma_1$.
We consider a geodesic triangle $[\gamma_1(0),\gamma_1(1)] \cup [\gamma_1(1),\gamma_2(1)] \cup [\gamma_2(1),\gamma_1(0)]$.
The geodesic triangle $\Delta(\xi(\gamma_1(0)), \xi(\gamma_1(1)), \xi(\gamma_2(1)))$ is a tripod. 
Let denote by $m \in T$, the center point of the tripod.
We divide the proof into the following cases: 
\begin{enumerate}
\renewcommand{\labelenumi}{\roman{enumi}).}
\item  $m \in L$.
\item  $m \in K$.
\end{enumerate}

The proof for case i) is analogous to that for case ii). We only consider  case ii). 
By construction, $\gamma_1$ and $\gamma_2$ can be expressed as follows;
There exist $p, a, b \in \xi^{-1}(\starNb(m))=X_m$ such that
\begin{itemize}
    \item $\gamma_1$ is the concatenation of $\Gamma(v,p)$, $\Gamma_m(p,a)$, and $\Gamma(a,w_1)$.
    \item $\gamma_2$ is the concatenation of $\Gamma(v,p)$, $\Gamma_m(p,b)$, and $\Gamma(b,w_1)$.
\end{itemize}
We set 
\begin{align*}
p=\gamma_1(c_0)=\gamma_2(c_0^{\prime}), \quad
a=\gamma_1(c_1), \quad 
b=\gamma_2(c_1^{\prime}),
\end{align*}
where $c_0, c_1, c_0^{\prime}, c_1^{\prime} \in [0,1]$.
\par
$X_m$ is a convex subspace of $Z$.
Note that $X_m$ admits a geodesic $(E,C)$-coarsely convex bicombing $\Gamma_m$.
$\Gamma$ restricted to $X_m \times X_m$ coincides with $\Gamma_m$. 
\par
The reparametrization of $\gamma_1|_{[0,c_0]}$ and $\gamma_2|_{[0,c_0']}$ coincide. 
Then (\ref{Ex-3-2}) of Definition~\ref{cftd} holds. 
\par
Thus, the image of $\gamma_1|_{[c_0,c_1]}$ coincides with that of $\Gamma(p_1,a)$
and the image of $\gamma_2|_{[c_0',c_1']}$ coincides with that of $\Gamma(p_2,b)$.
Item (\ref{Ex-3-4}) of Definition~\ref{cftd} holds.

\par
Set $t=d_{\mathsf{Z}}(\gamma_1(0),\gamma_1(1))$ and $s=d_{\mathsf{Z}}(\gamma_2(0),\gamma_2(1))$.
Since $\gamma_1$ and $\gamma_2$ are geodesic bicombings, we have $c_0t=c_0^{\prime}s$.
Without loss of generality, we may assume that $c_1 \leq c_1^{\prime}$.
Then, we have 
\begin{align*}
&c_1^{\prime} \cdot d_{\mathsf{Z}}(\gamma_1(1), \gamma_2(1)) - d_{\mathsf{Z}}(a,b) \\
&=c_1^{\prime}d_{\mathsf{Z}}(\gamma_1(1), \gamma_1(c_1)) + c_1^{\prime}d_{\mathsf{Z}}(\gamma_2(c_1^{\prime}), \gamma_2(1)) - (1-c_1^{\prime})d_{\mathsf{Z}}(a,b) \\
&=c_1^{\prime}(1-c_1)d_{\mathsf{Z}}(\gamma_1(1), \gamma_1(0)) + c_1^{\prime}(1-c_1^{\prime})d_{\mathsf{Z}}(\gamma_2(0), \gamma_2(1)) - (1-c_1^{\prime})d_{\mathsf{Z}}(a,b). 
\end{align*}
Moreover, by $c_1 \leq c_1^{\prime}$, the right-hand side of this equality can be estimated as follows:
\begin{align*}
& c_1^{\prime}(1-c_1)d_{\mathsf{Z}}(\gamma_1(1), \gamma_1(0)) + c_1^{\prime}(1-c_1^{\prime})d_{\mathsf{Z}}(\gamma_2(0), \gamma_2(1)) - (1-c_1^{\prime})d_{\mathsf{Z}}(a,b) \\
&\geq c_1^{\prime}(1-c_1^{\prime})d_{\mathsf{Z}}(\gamma_1(1), \gamma_1(0)) + c_1^{\prime}(1-c_1^{\prime})d_{\mathsf{Z}}(\gamma_2(0), \gamma_2(1)) - (1-c_1^{\prime})d_{\mathsf{Z}}(a,b) \\
&= (1-c_1^{\prime})\{ c_1^{\prime}d_{\mathsf{Z}}(\gamma_1(1), \gamma_1(0)) + c_1^{\prime}d_{\mathsf{Z}}(\gamma_2(0), \gamma_2(1)) - d_{\mathsf{Z}}(a,b)\} \\
&\geq (1-c_1^{\prime})\{ c_1d_{\mathsf{Z}}(\gamma_1(1), \gamma_1(0)) + c_1^{\prime}d_{\mathsf{Z}}(\gamma_2(0), \gamma_2(1)) - d_{\mathsf{Z}}(a,b)\} \\
&= (1-c_1^{\prime})\{ d_{\mathsf{Z}}(a, \gamma_1(0)) + d_{\mathsf{Z}}(\gamma_2(0),b) - d_{\mathsf{Z}}(a,b)\} \\
&\geq 0.
\end{align*}
The last inequality follows from the triangle inequality.
Then, we have 
\begin{equation}\label{TS-key}
 d_{\mathsf{Z}}(a,b) \leq c_1^{\prime} \cdot d_{\mathsf{Z}}(\gamma_1(1), \gamma_2(1)).
\end{equation}
Item (\ref{Ex-3-3}) of Definition~\ref{cftd} holds. 
\end{proof}

Lemma~\ref{tscbt} can be proved in the same manner as Lemma~\ref{tscft}.

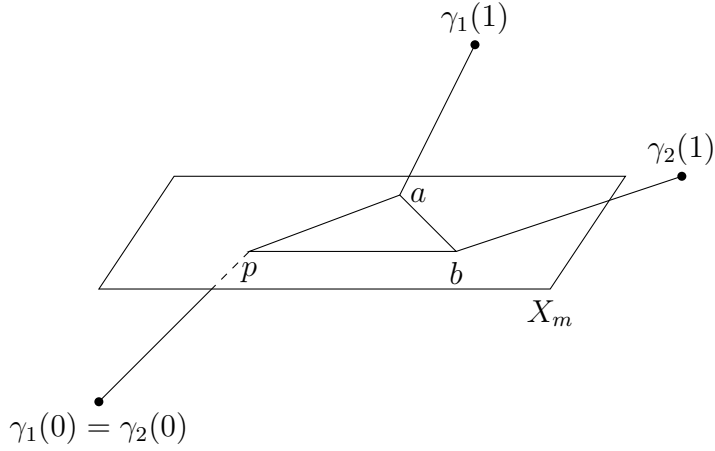
\begin{figure}[h]
\begin{tikzpicture}
\draw(0,0)--++(6,0)--++(-1,-1.5)--++(-6,0)--cycle; 
\draw(5,-1.5)node[below]{$X_m$};
\draw[dashed](1,-1)--++(-0.5,-0.5);
\draw(0.5,-1.5)--++(-1.5,-1.5);
\draw(-1,-3)node[below]{$\gamma_1(0)=\gamma_2(0)$};
\fill(-1,-3)circle(0.06);
\draw(1,-1)--++(2,0.75)--++(0.75,-0.75)--cycle;
\draw(1,-1)node[below]{$p$};
\draw(3.5,-0.25)node[left]{$a$};
\draw(3.75,-1)node[below]{$b$};
\draw(3,-0.25)--++(1,2);
\draw(4,1.75)node[above]{$\gamma_1(1)$};
\fill(4,1.75)circle(0.06);
\draw(3.75,-1)--++(3,1);
\draw(6.75,0)node[above]{$\gamma_2(1)$};
\fill(6.75,0)circle(0.06);
\end{tikzpicture}
\caption{An example of geodesic triangles of $\mathsf{Z}$.}
\label{TS-gt}
\end{figure}

\subsection{Combination of equivariant bicombings}
In this subsection, we prove part \eqref{TS-main-G-equiv} of Theorem \ref{TS-Main_thm}. From now, we suppose $\xi \colon Z \to T$ is a $G$-tree of spaces
and the family of bicombings $\{{\Gamma}_{k}\mid k\in K\}$ is $G$-equivariant. Let $v,w \in Z$. 
By the construction of $\Gamma(v,w)$, 
there exists a sequence of distinct elements $\{v_1,v_2,\dots,v_m\} \subset \xi^{-1}(L)$ such that if we let $v_0=v$ and $v_{m+1}=w$, then 
\begin{itemize}
\item there exists a sequence of distinct elements $\{k_i\}_{i=0}^m \subset K$ such that $v_i, v_{i+1} \in X_{k_i}$,
\item $\Gamma(v,w)$ is obtained by concatenating ${\Gamma}_{k_i}(v_i,v_{i+1})$ at $\{v_i\}_{i=1}^{m-1}$.  
\end{itemize}

As in equality \eqref{eq:GammaUV}, the geodesic segment $\Gamma(v,w)(\cdot) : [0,1] \to \mathsf{Z}$ is represented by
\begin{align*}
\Gamma(v,w)(t)=
{\Gamma}_{k_i}(v_i,v_{i+1})\left( \left( t-\dfrac{t_i}{t_{m+1}}\right ) \cdot \dfrac{t_{m+1}}{t_{i+1}-t_i} \right)\  \text{for}\ \dfrac{t_i}{t_{m+1}} \leq t \leq \dfrac{t_{i+1}}{t_{m+1}}
\end{align*}
where $t_0=0$, and $t_i=\displaystyle\sum_{l=0}^{i-1}d_{\mathsf{Z}}(v_l,v_{l+1})$.

On the other hand, $\Gamma(gv,gw)$ is represented by 
\begin{align*}
\Gamma(gv,gw)(t)=
{\Gamma}_{gk_i}(gv_i,gv_{i+1})\left( \left( t-\dfrac{t_i}{t_{m+1}}\right ) \cdot \dfrac{t_{m+1}}{t_{i+1}-t_i} \right)\ 
\text{for}\ \dfrac{t_i}{t_{m+1}} \leq t \leq \dfrac{t_{i+1}}{t_{m+1}},  
\end{align*}
since $\xi$ is a $G$-map, $L$ is a $G$-invariant subset of $T$, and $ 
t_i=\displaystyle\sum_{l=0}^{i-1}d_{\mathsf{Z}}(gv_l,gv_{l+1})$, where this last equality is a consequence of the $G$-action on $Z$ being by isometries.

Since the family $\{{\Gamma}_{k}\}_{k\in K}$ is $G$-equivariant,
for $\dfrac{t_i}{t_{m+1}} \leq t \leq \dfrac{t_{i+1}}{t_{m+1}}$, 
we have
\begin{align*}
\Gamma(gv,gw)(t)
&={\Gamma}_{gk_i}(gv_i,gv_{i+1})\left( \left( t-\dfrac{t_i}{t_{m+1}}\right ) \cdot \dfrac{t_{m+1}}{t_{i+1}-t_i} \right) \\
&=g \cdot {\Gamma}_{k_i}(v_i,v_{i+1})\left( \left( t-\dfrac{t_i}{t_{m+1}}\right ) \cdot \dfrac{t_{m+1}}{t_{i+1}-t_i} \right) \\
&=g \cdot \Gamma(v,w)(t). 
\end{align*}
Therefore, the geodesic bicombing $\Gamma: \mathsf{Z} \times \mathsf{Z} \times [0,1] \to \mathsf{Z}$ is $G$-equivariant.

%% file: 60-Proofs.tex
\begin{proof}[Proof of Theorem~\ref{thmx:Io}]
Let $(G_i, \mathcal{H}_i\cup \{K_i\})$ for $i=1,2$ be group pairs, and
$G=G_1\ast G_2$. The arguments for $\mathcal{A}_{wsh}$ and
$\mathcal{A}_{sh}$ are completely analogous, we prove the latter.
Suppose that $(G_i, \mathcal{H}_i\cup \{K_i\})$ is a semi-hyperbolic
group pair for $i=1,2$.  Then the pair $(G_i, \mathcal{H}_i\cup
\{K_i\})$ admits a Cayley-Abels graph $X_i$ with a bounded
$G_i$-equivariant quasi-geodesic bicombing $\hat\Gamma_i$. By definition
of Cayley-Abels graph, there exists a vertex $x_i\in \Gamma_i$ with
$G_i$-stabilizer equal to $K_i$.  Applying
Theorem~\ref{thm:CombinationFine} (or~\cite[Theorem 3.1]{BiMa23}) to
$X_1$, $X_2$ and the vertices $x_1,x_2$, we obtain a Cayley-Abels graph
$Z$ for the pair $(G_1\ast_CG_2, \mathcal{H}\cup\{\langle K_1,K_2
\rangle\})$, and a continuous map $\xi\colon Z \to T$ which is a
$G$-tree of spaces with the property that $\xi^{-1}(\xi(x_i))$ is
isomorphic to $X_i$. It follows that each vertex space of $Z$ admits a
bounded quasi-geodesic bicombing and this family of bicombings is
$G$-equivariant. By Theorem~\ref{TS-Main_thm}, the graph $Z$ admits a
$G$-equivariant bounded quasi-geodesic bicombing. Therefore, the pair
$(G,\mathcal{H}_1\cup\mathcal{H}_2\cup\{\langle K_1,K_2 \rangle\})$ is a
semi-hyperbolic group pair.

Now we prove the statement for $\mathcal{A}_{gcc}$.
Suppose that $(G_i, \mathcal{H}_i\cup \{K_i\})$  is a geodesic coarsely convex group pair for $i=1,2$.
Each group $G_i$ acts geometrically relative to $\mathcal{H}_i\cup \{K_i\}$  on a geodesic coarsely convex space $X_i$ satisfying the conclusion 
of~\cref{prop:extra-condition}. Let $x_i\in X_i$ such that its $G_i$-isotropy is the subgroup $K_i$.
Applying   Theorem~\ref{thm:CombinationFine} to the spaces $X_1,X_2$ and points $x_1,x_2$, 
it follows that  there exists a cocompact geodesic metric $G$-space $Z$ such that for each $H\in \mathcal{H}_1\cup\mathcal{H}_2\cup\{\langle   K_1,K_2  \rangle \}$ the fixed point set $Z^H$ is non-empty and bounded, and for each $y\in Z$ the isotropy $G_z$ is either a finite subgroup or a conjugate of some $H\in \mathcal{H}_1\cup\mathcal{H}_2\cup\{\langle   K_1,K_2  \rangle \}$.
 Moreover, there is a continuous map $\xi\colon Z \to T$ which is a $G$-tree of spaces with the property that $\xi^{-1}(\mathsf{Star}(\xi( x_i)))$ is $\spike(X_i, C, x_i)$. 

 By \cref{remark:conedoff-preserve-gcc},
 each vertex space $\xi^{-1}(\mathsf{Star}(k))$ of $Z$  admits a geodesic coarsely convex bicombing  and this family of bicombings is $G$-equivariant. Then, Theorem~\ref{TS-Main_thm} implies that the geodesic metric $G$-space $Z$ admits a $G$-equivariant geodesic coarsely convex bicombing. This proves that the group pair  $(G,\mathcal{H}_1\cup\mathcal{H}_2\cup\{\langle K_1,K_2 \rangle\})$ is geodesically coarsely convex.
\end{proof}

\begin{proof}[Proof of Theorem~\ref{thmx:IIIo} ]
 The argument is completely analogous to the proof of Theorem~\ref{thmx:Io},
 the main  difference is the use the version of Theorem~\ref{thm:CombinationFine} for HNN-extensions which is Theorem~\ref{thm:HNNFine}. 
\end{proof}